\documentclass[11pt]{amsart}
 \usepackage[foot]{amsaddr}
\usepackage{bbm,bm}
\usepackage{amssymb}
\usepackage[margin=1.5in, total={6.0in, 8in}]{geometry}
\usepackage{colordvi}
\usepackage{graphicx,xspace}
\usepackage{color}
\usepackage{xspace}
\usepackage{comment}

\usepackage[bookmarks=false,colorlinks,linkcolor=blue,citecolor=green]{hyperref}
\definecolor{mains}{cmyk}{.3, .85, .75, 0}  %unsolved problem box color inside
\definecolor{afb}{rgb}{0.03, 0.27, 0.49}
\definecolor{def}{rgb}{0.27, 0.03, 0.49}

\newcounter{FNC}[page]
\def\fauxfootnote#1{{\addtocounter{FNC}{2}$^\fnsymbol{FNC}$%
     \let\thefootnote\relax\footnotetext{$^\fnsymbol{FNC}$\Magenta{#1}}}}
\numberwithin{equation}{section}
\newtheorem{theorem}{Theorem}[section]

\newtheorem{lemma}[theorem]{Lemma}

\newtheorem{thm}[theorem]{Theorem}
\newtheorem{defn}[theorem]{Definition}
\newtheorem{exm}[theorem]{Example}
\newtheorem{rem}[theorem]{Remark}
\newenvironment{definition}[1][]{\rm\begin{defn}[#1]\rm}{\end{defn}}
\newenvironment{example}[1][]{\rm\begin{exm}[#1]\rm}{\end{exm}}
\newenvironment{remark}[1][]{\rm\begin{rem}[#1]\rm}{\end{rem}}

\author{Stefan Forcey} \address[S. Forcey]{
    Department of Mathematics\\
    The University of Akron\\
    Akron, OH 44325-4002
    }
    \email{sforcey@uakron.edu}  \urladdr{http://www.math.uakron.edu/\~{}sf34/}

\title[Kalmanson-Kron Reconstruction]{Circular planar electrical networks, Split systems, and Phylogenetic networks.}

\begin{document}

\begin{abstract}

We study a new invariant of circular planar electrical networks, well known to phylogeneticists: the circular split system. We use our invariant to answer some open questions about levels of complexity of networks and their related Kalmanson metrics. The key to our analysis is the realization that certain matrices arising from weighted split systems are studied in another guise: the Kron reductions of Laplacian matrices of planar electrical networks.  Specifically we show that a response matrix of a circular planar electrical network corresponds to a unique resistance metric obeying the Kalmanson condition, and thus a unique weighted circular split system. Our results allow interchange of methods: phylogenetic reconstruction using theorems about electrical networks, and circuit reconstruction using phylogenetic techniques.  %Thus the space $\Omega(n)$ of reponse matrices forms a subspace of CSN($n$), the space of circular split networks. 

 \end{abstract}

\maketitle

%%%%%%%%%%%%%%%%%%%%%%%%%%%%%%%%%%%%%%%%%%%%%%%%%%%%%%%%%%%%%%%%%%%%%%%%%%%%%%%%%%%%%

 \begin{figure}[h!]
    \centering
    \includegraphics[width=\textwidth]{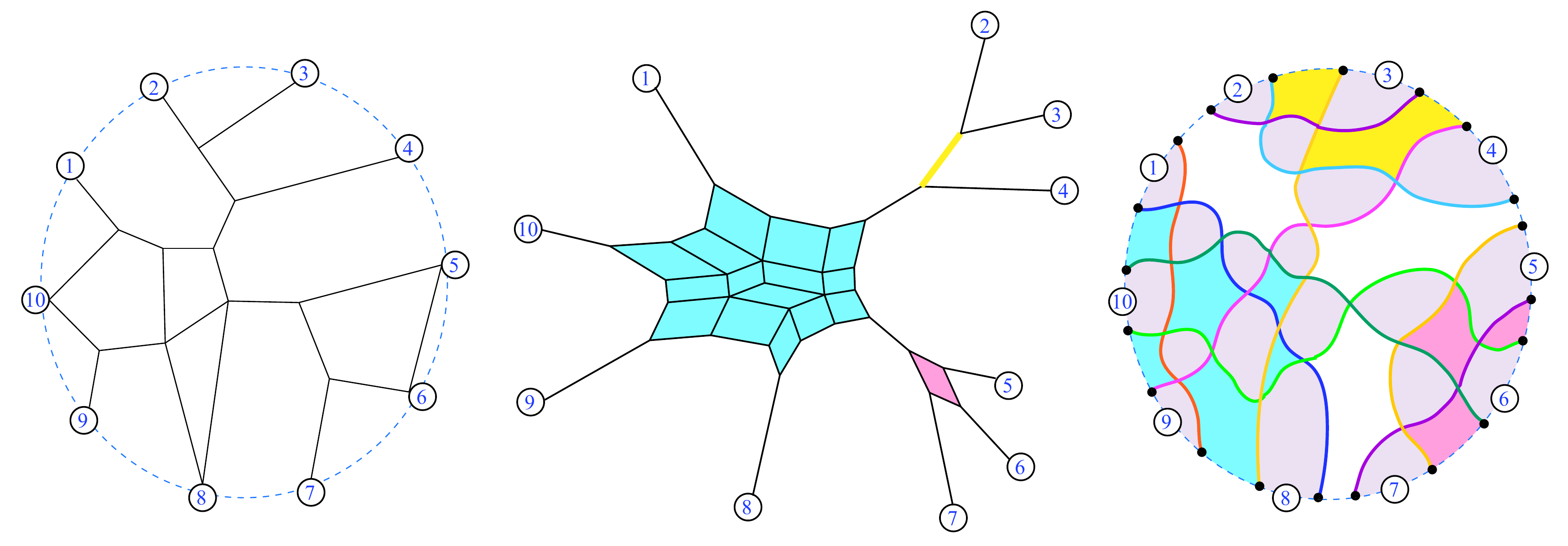}
    \caption{On the left, the graph of a circular planar network. Center, the graph of its associated split system: the new invariant of the network. At right, the strand matching diagram corresponding to the network; compare to Example~\ref{big}.}
    \label{covernet}
\end{figure}

%\tableofcontents

\section{Introduction}

 Suppose that we are given a black box containing a connected electrical network $N$, made of many tangled wires. On the surface of the box are $n$ exposed terminals, and we can test these by applying voltage to them. We apply a voltage of 1 to each of the terminals in turn, each time using the remaining $n-1$ terminals to complete the circuit in parallel. We record the results as a matrix $M.$ This  \emph{response matrix} has entries $M_{i,j}$ equalling the current at terminal $j$ when  the unit voltage (our battery) is applied to terminal $i$,
and voltage of 0 (grounding) to all other terminals simultaneously. Alternatively, we could use an ohmmeter and test the effective resistance (impedance) between pairs of our terminals. We record these results as the \emph{resistance matrix} $W$, where $W_{i,j}$ is the effective resistance between terminal $i$ and $j.$ For a single wire, the conductance is the reciprocal of the resistance. For the entire circuit the formulas relating $M$ and $W$ are more complicated: see section~\ref{formulas}.\\

The inverse problem for electrical networks is to try to reconstruct a network $N$ using the response matrix $M$ on a set of terminals labeled $\{1,\dots,n\} = [n].$ Many solutions are typically possible. In their monograph \cite{curtisbook}, Curtis and Morrow completely solve the inverse problem for circular planar graphs with a given circular order: if a response matrix $M$ obeys the condition of non-negative circular minors, then they show how to reconstruct both a graph of a planar network and the conductances of the edges of that network that provide the desired $M$. In \cite{dorfler} Dorfler and Bullo show a  simple solution when the network $N$ is connected: regardless of the actual edge structure of $N$, the matrix $M$ will have the form of a weighted Laplacian of a weighted graph $K(N)$ on the vertex set $[n].$ This graph is made of cliques, and is indeed a network which has the response matrix $M.$   %The inverse problem  for circular planar networks can also be seen as a design problem: given the desired response matrix, can we implement it on a printed circuit board? \\
\\

Meanwhile, recent work in phylogenetics  has focused on very similar reconstruction problems, as seen in \cite{huber2021} and \cite{frontiers}. Phylogenetic networks are combinatorially exactly the same objects as 
electrical networks, although the edge weights are decidedly less well understood. The main difference in the two fields has been, historically, that electric networks are weighted with conductance while phylogenetic networks are weighted with statistical distance metrics. However there is a simple bridge between the paradigms: the genetic distance between two extant taxa (species or individuals) can be  analogous to the resistance between two exposed terminals. 
In papers such as \cite{huber2021} and \cite{durell} the theoretical distance is assumed to be the length of the shortest path, length of the longest path, or a multiset of lengths of paths. In general, the resistance distance captures more information than any (multi)set of path lengths. In \cite{frontiers} the authors show how traditional distances measuring genetic distance, such as the Jukes-Cantor distance, can be seen as resistance distances if the recombination events obey certain expectations. Under those severe assumptions our results here allow the reconstruction of phylogenetic networks as well as electrical. However, there is much work to be done to clarify when such assumptions are justified. \\

Since the work of Kron it has been known that response matrices $M$ are Schur complements of the weighted Laplacian of an underlying graph of $N$ using conductances as weights \cite{kron4, kron3, kron2, kron1}. Curtis, Ingerman, and Morrow gave a complete characterization of the response matrices for networks that can be embedded in a plane with the terminals $\{1,\dots,n\}$ in counting order around their exterior \cite{curtis1, curtisbook}. These circular planar networks are the ones we focus on in this paper. If the desired circular order is not predetermined, there is a brute-force method to look for it:  reordering the terminals in all possible circular orders, rearranging the entries of the response matrix to correspond, and rechecking the circular minors for non-negativity. The first consequence of our new result is that fast algorithms for Kalmanson metrics can replace that brute-force process to quickly recover a candidate circular order for planarity, or reveal that it does not exist. Further, our results show that when measured genetic distances are found to be Kalmanson, and their response matrix is mathematically electrical, we can use Curtis and Morrow's algorithms on subnetworks to completely reconstruct the phylogenetic network. \\

Kenyon and Wilson in \cite{kenyon} recently described an alternate algorithm for finding the conductances of the edges of a standard network $N$ for each equivalence class of circular planar networks, if the underlying graph of $N$ is known. They mention the open question of finding the underlying graph in a more efficient way than exhibited by Curtis and Morrow. Here we show that in simple cases the underlying graph of $N$ is recoverable quickly from the response matrix $M$, using  mathematics that was developed in the field of phylogenetics. Our main theorem allows us to recover some basic features (circular planar ordering, bridges) of the graph of $N$ for all $n,$ and for some cases we can refine that recovery to find the complete structure. Combining our methods with the algorithm of Curtis and Morrow allows for faster graph reconstruction. In the other direction, our results show that complete phylogenetic network reconstruction is possible using Curtis and Morrow's algorithm. 
In the phylogenetic reconstruction problem planarity is not always a concern. However, asking for planarity can be seen as a way to insist on parsimony, as planar networks are the next step up in complexity from trees.  Indeed 1-nested and 2-nested phylogenetic networks, as defined in \cite{gambette-huber}, are both circular planar. 

\subsection{Acknowledgements}

Thanks to Satyan Devadoss for taking the time for many, many  conversations, and to Richard Kenyon, Thomas Lam and Jim Stasheff for answering lots of questions.

\section{Outline and Results}

In Section~\ref{nec} we go over the definitions and prior results that are needed to make this paper self-contained. Then in Section~\ref{splitsec} we prove our main result: Theorem~\ref{bigth} says that a response matrix for any circular planar electrical network corresponds uniquely to a Kalmanson metric (the resistance matrix), which in turn corresponds uniquely to a circular split system. This result settles a conjecture from \cite{frontiers}. A further conjecture from \cite{frontiers} is shown not to hold in Theorem~\ref{sym}. Practical implications of the main theorem are immediate: see Theorems~\ref{brij}  and~\ref{brijj}. We use those results to solve inverse problems in Section~\ref{recon}.  We can take a response matrix, test to see if it could have arisen from a circular planar network, and then find its circular split system.  The split system, visualized as a split network, yields a first look at the bridge structure of the unknown network.  Inside the bridge-free portions of the network we can reconstruct the local graphs---isolating the bridge-free portions allows us to reduce the big problem to several smaller ones. In Section~\ref{new} we discuss the spaces of networks and some open questions.  %Given a bridge-free portion, connected by $k$ bridges to $k$ parts of the network, we choose a terminal attached to each of those $k$ parts and take the Kron reduction with respect to that new smaller boundary. Finding the corresponding underlying graph for the smaller $k$-terminal network will answer the question of edges in that bridge-free portion of the big network.

\section{Preliminaries}\label{nec}
 
Most of the material in this section is an abbreviated review of concepts covered in two monographs about circular planar electrical and phylogenetic networks, respectively  by Curtis and Morrow \cite{curtisbook} and Steel \cite{steelphyl}. See those for complete definitions and many more theorems. We also make connections with the terminology from more recent papers, such as \cite{bapat2}, \cite{dorfler}, and \cite{huber2021}. 

 Abstractly, an \emph{electrical network} is a  graph $N$ with $m$ vertices, with $n$ of those vertices labeled by $[n] = \{1,\dots,n\}$ and called the \emph{boundary} nodes, or terminals, and the remaining $m-n$ vertices labeled by $\{n+1,\dots,m\}$ called interior nodes.  The edges of $N$ are usually given non-negative weights (often positive) that represent conductance. When the boundary nodes are labeled by $[n]$ in a clockwise circle (usually in counting order, but in general we may use an arbitrary circular order), and the graph $N$ can be drawn in the disk bounded by that circle with no crossed edges, we call $N$ a \emph{circular planar electrical network}. In this case the interior nodes can be left unlabeled, or labeled arbitrarily if that is handy. 
 
 A \emph{phylogenetic network} is a  graph $N$ with $m$ vertices, with $n$ of those vertices labeled by $[n] = \{1,\dots,n\}$ and called the \emph{leaf} nodes, or taxa, and the remaining $m-n$ vertices labeled by $\{n+1,\dots,m\}$ called interior nodes.  The edges of $N$ are usually given non-negative weights (often positive) that represent genetic mutation distance. When the leaf nodes are labeled by $[n]$ in a clockwise circle (usually in counting order, but in general we may use an arbitrary circular order), and the graph $N$ can be drawn in the disk bounded by that circle with no crossed edges, we call $N$ a \emph{circular planar phylogenetic network}. In this case the interior nodes can be left unlabeled, or labeled arbitrarily if that is handy. Sometimes there are two simplifying requirements: that the leaves be degree 1 and that there be no degree 2 nodes.
 
 Notice that the definitions of electrical and phylogenetic circular planar networks differ only in semantics. Another difference in their traditional treatment is that circular planar electrical networks are usually considered to be defined as having a given circular order of terminals, typically clockwise in counting order. This makes physical sense, for instance if the network will be printed on the surface of a circuit board. In contrast, a phylogenetic network $N$ will not be constructed in physical space, so it is considered to be planar with respect to a set of possible circular orders found by twisting $N$ around cut-vertices and bridges of the network, rotating, and flipping. All such circular orders $c$ of $[n]$ that allow the network to be drawn in the plane are called \emph{consistent} with $N.$ All such re-drawings of the phylogenetic network are equivalent. We can easily extend such freedoms to an electrical network; $N$ can be circular planar with respect to some circular orderings $c$ but not with respect to others. We will say that $c$ is \emph{consistent} with electrical $N$ if $N$ is circular planar with its terminals in the circular order $c.$
 
 The (weighted) graph Laplacian $L(N)$ is an $m\times m$ symmetric matrix. Off-diagonal entry $L_{ij}$ is the conductance of the edge $\{i,j\}$ or 0 if there is no such edge. Diagonal entry $L_{ii}$ is the negative value which makes the row and column both sum to zero, that is, $L_{ii}$ is the negative of the sum of the edges adjacent to vertex $i$. (This sign convention is sometimes reversed, as in \cite{curtisbook}
where only the diagonal entries are positive.) \\

    \begin{figure}[h]
     \centering
     \includegraphics[width=\textwidth]{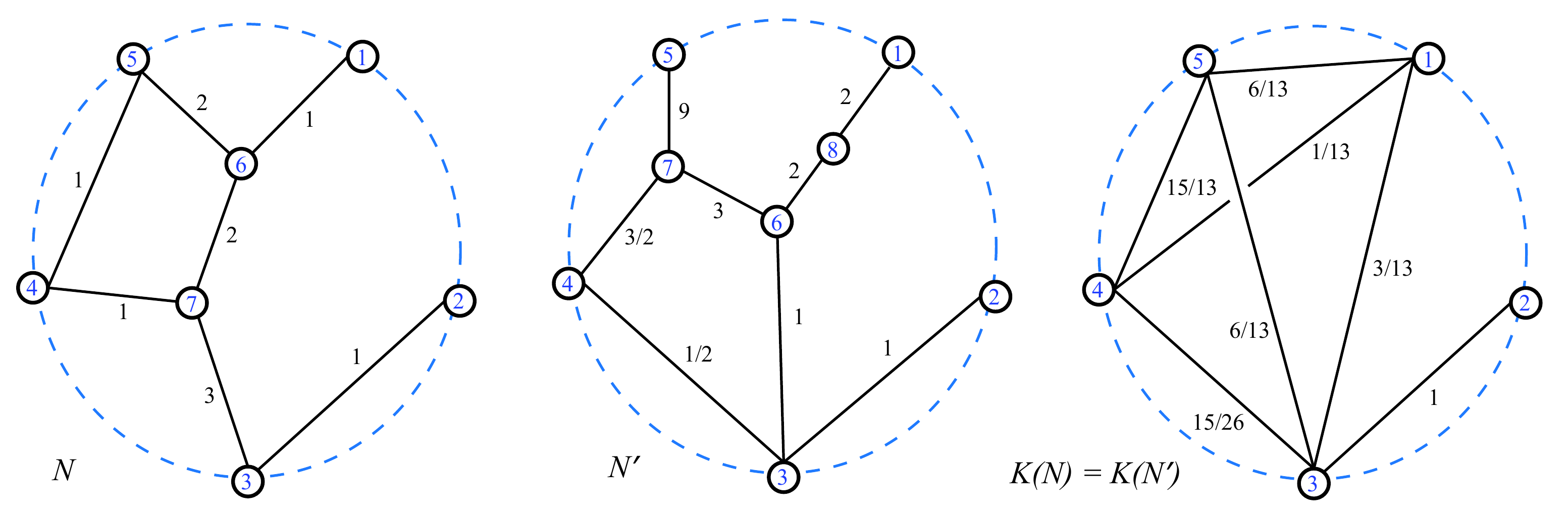}
     \caption{The circular planar electrical network on the left is equivalent to the one in the center, $N\sim N'.$ Their shared Kron reduction is on the right. }
     \label{exrun1pt5}
 \end{figure}
 
 \subsection{Matrix formulas}\label{formulas}
  The \emph{Kron reduction} of a graph Laplacian $L(N)$, given a choice of nodes $[n]$ to be the boundary (leaves), is the Schur complement with respect to the non-boundary nodes. That is, letting $A$ be the submatrix of $L$ using rows and columns $1,\dots,n$, letting $B$ be the submatrix of $L$ using rows $1,\dots,n$ and columns $n+1,\dots,m$, and letting $C$ be the submatrix of $L$ using rows and columns $n+1,\dots,m$, then 
 $$
 M(N)=A-BC^{-1}B^T
 $$
 Note that the response matrix of any network $N$ whose nodes are all selected as boundary nodes is the same as $L(N)$, the Laplacian.  The Kron reduction with respect to non-boundary nodes is precisely the response matrix $M(N)$ of the network, and has also been termed the Dirichlet-to-Neumann map, \cite{kenyon} and \cite{verdiere1}.  Two electrical networks with boundary nodes $[n]$ are \emph{electrically equivalent}, $N\sim N'$, when they have the same response matrix: $M(N)= M(N')$.

 The resistance matrix for a network $N$ whose nodes are all selected as boundary nodes is called  $R(N)$, also known as the resistance metric. Note that Kron reduction (using a Schur complement) of $L(N)$ corresponds to simply restricting $R(N)$, that is, finding the submatrix $W$ of $R$ for the new set of boundary nodes. 
 
For $N$ with boundary nodes $[n]$, there is a well-known one-to-one mapping between the response matrix  $M=M(N)$ of a network and the corresponding resistance distance matrix  $W=W(N)$. The following formulas are found in Lemma 3.11 of \cite{dorfler}, and also (in variant forms) in the studies of resistance metrics on graphs, in \cite{klein93}, \cite{bapat1}, \cite{bapat2}.
\begin{lemma}
Let $X^{\dagger}$ denote the pseudoinverse of $X.$ Let $X_D$ denote the diagonal matrix formed from $X$ by keeping only the main diagonal of $X$ (setting off-diagonal entries to zero). Let $J$ be the $n\times n$ matrix whose entries are all 1.  Then, given $M$ we find  $W:$\end{lemma}
$$W(M) = \left((-M)^{\dagger}\right)_DJ+J\left((-M)^{\dagger}\right)_D-2(-M)^{\dagger} $$
and given $W$ we find $M:$
$$M(W) = \Bigg(\frac{1}{2}\left(W-\frac{1}{n}(WJ+JW)+\frac{\text{trace}(WJ)}{n^2}J\right)\Bigg)^{\dagger}.$$

Note that $L=L(N)$ and $R=R(N)$ are related by the same formulas are $M$ and $W$, since $L$ and $R$ are the response and resistance matrices for the case when all nodes are considered to be terminals. 

\begin{example}\label{mats}
In Figure~\ref{exrun1pt5} we show a circular planar electrical network $N.$  Here are the matrices calculated from $N$:
$$
L(N)=\begin{bmatrix}
-1 &  0 &  0 &  0  & 0  & 1 &  0\\
0 &  -1 &  1 &  0 &  0 &  0 &  0\\
0 &  1 &  -4 &  0 &  0 &  0 &  3\\
0 &  0 &  0 &  -2 &  1 &  0 &  1\\
0 &  0 &  0 &  1 &  -3 &  2 &  0\\
1 &  0 &  0 &  0 &  2 &  -5 &  2\\
0 &  0 &  3 &  1 &  0 &  2 &  -6
\end{bmatrix}
~
R(N) = \begin{bmatrix}
     0 & 11/4 & 7/4 & 7/4 & 17/12 & 1 & 17/12\\    
      11/4 & 0 &  1 &  2 &  2 &  7/4 & 4/3\\     
       7/4 & 1 &  0 &  1 &  1 &  3/4 & 1/3 \\    
       7/4 & 2 &  1 &  0 &  2/3 & 3/4 & 2/3 \\    
      17/12 & 2 &  1 &  2/3 & 0 &  5/12 & 2/3  \\   
       1 &  7/4 & 3/4 & 3/4 & 5/12 & 0 &  5/12 \\   
      17/12 & 4/3 & 1/3 & 2/3 & 2/3 & 5/12 & 0
\end{bmatrix}
$$

  $$
  A=\begin{bmatrix}
-1 &  0 &  0 &  0 &  0\\
0 &  -1 &  1 &  0 &  0\\
0 &  1 &  -4 &  0 &  0\\
0 &  0 &  0 &  -2 &  1\\
0 &  0 &  0 &  1 &  -3
\end{bmatrix}
 ~
 B=
 \begin{bmatrix}
1 &  0\\
 0 &  0\\
 0 &  3\\
 0 &  1\\
 2 &  0
 \end{bmatrix}
 ~
 C=
 \begin{bmatrix}
 -5 &  2\\
2 &  -6
 \end{bmatrix}
  $$
  
$$
M(N) = \begin{bmatrix}
     -10/13 & 0 & 3/13 & 1/13 & 6/13\\
       0 & -1 & 1 & 0 & 0\\
       3/13 & 1 & -59/26 & 15/26 & 6/13\\
       1/13 & 0 & 15/26 & -47/26 & 15/13\\
       6/13 & 0 & 6/13 & 15/13 & -27/13
 \end{bmatrix}
~
W(N) =\begin{bmatrix}
    0 & 11/4 & 7/4 & 7/4 & 17/12\\
      11/4 & 0 & 1 & 2 & 2\\
       7/4 & 1 & 0 & 1 & 1\\
       7/4 & 2 & 1 & 0 & 2/3\\
      17/12 & 2 & 1 & 2/3 & 0
\end{bmatrix}
$$
\end{example}
 
 The response matrix $M$ is also a weighted graph Laplacian, for a new graph of $n$ nodes. This new weighted graph is called the  Kron reduced network $K(N),$ and it is described in Theorem 3.4 of \cite{dorfler}. The only vertices of $K(N)$ are the boundary nodes $[n].$ Edges between boundary nodes in $N$ are repeated in $K(N)$. In general there is an edge $\{i,j\}$ in $K(N)$ if and only if those boundary nodes are connected by a  path in $N$ not containing any other boundary nodes. Figure~\ref{exrun1pt5} shows $N$ on the left and its Kron reduced network  $K(N)$ on the right.\\

\subsection{Connections and minors} For the following concepts, see Example~\ref{circ}. A \emph{circular pair} of $[n]$ is a pair of disjoint ordered lists $(P,Q)= (p_1,\dots,p_k;q_1,\dots,q_k)$ of elements of $[n]$,  such that  we can write the list $P$, followed by the list $Q$ in reverse $(p_1,\dots,p_k,q_k,\dots,q_1)$, as a concatenated list of length $2k$ that respects the circular order, without looping.  That is, the two sets are \emph{non-interlaced}, or \emph{non-crossing} on the circle.  In the $n\times n$ response matrix $M$ a \emph{circular submatrix} of size $k$ associated to a circular pair $(P,Q)$ is the $k\times k$ matrix made by selecting the rows of $M$ listed in $P$, and then the columns of $M$ listed in $Q$. Rows and columns of the submatrix keep the respective original orders given by the two lists in $(P,Q)$.

The \emph{circular minor} of $M(N)$ associated to the circular pair $(P,Q)$ of size $k$ is the determinant of the circular submatrix. We denote it by $\text{det}\, M(P,Q).$ Note that in \cite{curtis0} their matrix $M$ is the negative of ours (their off-diagonal entries are negative) so they multiply their determinant by $(-1)^k$, while we do not. \\ 

The following is Theorem 3 by Curtis, Ingerman, and Morrow  in \cite{curtis2}.\\

\begin{thm}
A response matrix $M$ for an electrical network $N$ has all non-negative circular minors if and only if  $N$ is circular planar.
\end{thm}

A \emph{k-connection} associated to a circular pair $(P,Q)$ of a circular planar electrical network $N$ is a set of $k$ non-intersecting interior paths in $N$ from the $k$ terminals $p_i$ listed in $P$ to the corresponding $k$ terminals $q_i$ listed in $Q.$ That is, each of the $k$ paths has only interior nodes (except for that path's terminal endpoints) and no interior node is in more than one of the paths.

\begin{lemma}
For a circular planar electrical network $N$, a circular minor of $M(N)$ associated to the circular pair $(P,Q)$  of size $k$ obeys $\text{det}\, M(P,Q) > 0$ if and only if there exists in (planar) $N$ a $k$-connection for $(P,Q).$
\end{lemma}

\begin{example}\label{circ}
The network
 $N$ in Figure~\ref{exrun1pt5} is circular planar, and $N\sim N'.$ The 1-connections are the 14 non-zero off-diagonal entries of $M$ and there are five 2-connections.  For instance, by inspection of the paths in the network from 1 to 5 and from 3 to 4, (1,3;5,4) is a 2-connection. As expected, its circular minor is non-zero: $$det\begin{bmatrix}
 6/13 & 1/13\\
 6/13 & 15/26
 \end{bmatrix} = 39/169.$$ In contrast (2,3;1,4) is not a 2-connection, reflected by the fact that its circular minor is  $$det\begin{bmatrix}
 0 & 0\\
 3/13 & 15/26
 \end{bmatrix} = 0.$$
\end{example}

\subsection{Invariants} The matrix $M(N)$ is by definition an algebraic invariant of electrical equivalence, and thus also is the matrix $W(N)$. There are several combinatorial invariants of electrical equivalence for circular planar electrical networks. The first three are as follows: if $N\sim N'$ then both have the same \emph{set of connections},  both have the same \emph{perfect matching} on $[2n],$ and both have the same set of \emph{minimal graphs}. Minimal graphs are called \emph{critical} in \cite{curtisbook}, and are defined to be such that deletion or contraction of any edge decreases the set of connections. In Figure~\ref{exrun1pt5}, the first network $N$ is minimal (critical), while the equivalent network $N'$ is not. In fact, for a given $n$ the three sets of images (of these three invariants) are in mutual bijection and detect the same electrical equivalences. Each minimal graph has the same number of edges, which also turns out to be the dimenension of the subspace of response matrices that map to it. More about minimal graphs, strand matchings, and spaces can be found in Theorem  2.3 of \cite{kenyon},  and also in \cite{verdiere2, curtisbook}.  We will use the following Example~\ref{match} to explain how to find the perfect matchings corresponding to a circular planar network.

\begin{example}\label{match}
In Figure~\ref{exrun3} the network $N$ is shown with its medial graph (dashed) and stubs $1,\dots,10$ (often labeled as $t_i$) numbered clockwise starting just counterclockwise of terminal 1. The \emph{strands} are formed as paths in the medial graph which do not turn left or right at the vertices of the medial graph. If pairs of strands cross each other at most once (they are called lens-free), then the network picture was minimal and the strands give the associated perfect matching: here we have the matching $\{\{1,7\},\{2,8\},\{3,5\},\{4,9\},\{6,10\}\}.$ Note: as shown in \cite{curtisbook}, if any pair of strands crosses each other more than once (exhibits a lens), then the original network was not critical. 

    \begin{figure}[h]
     \centering
     \includegraphics[width=\textwidth]{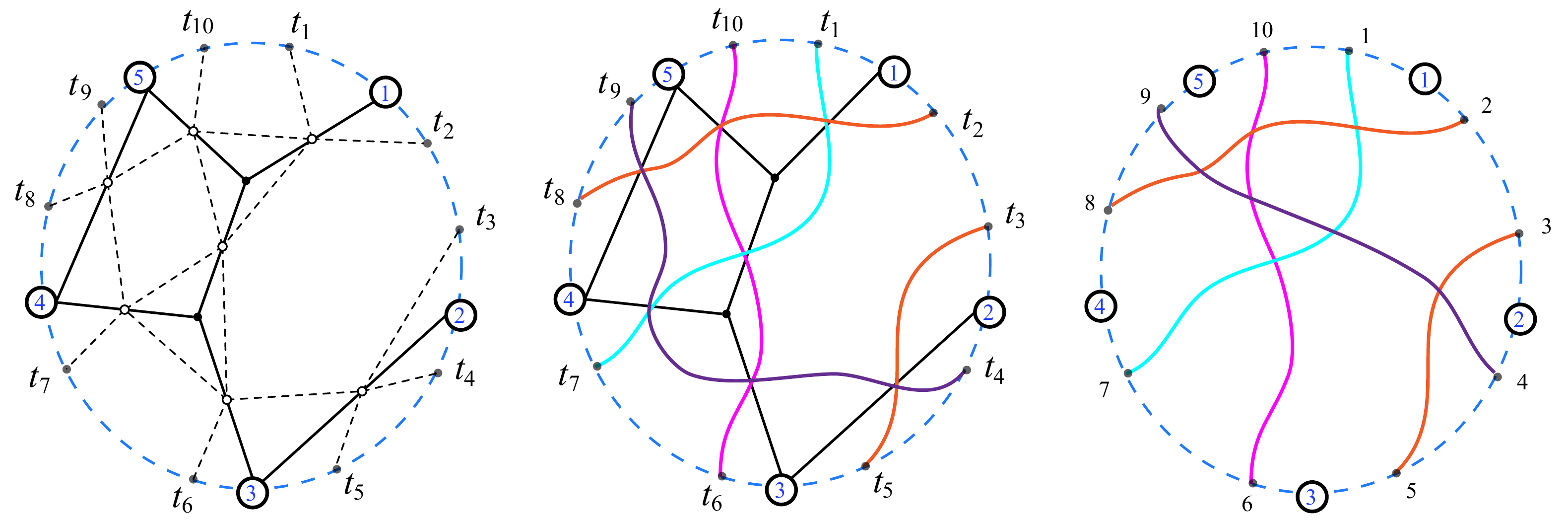}
     \caption{On the left the graph of the network $N$ from Figure~\ref{exrun1pt5} is shown with its medial graph dashed. In the center the medial graph has been colored to see the strands, and the lens-free matching via strands is shown on the right. }
     \label{exrun3}
 \end{figure}

\end{example}

A \emph{split} $A|B$ of $[n]$ is a partition of $[n]$ into two parts $A$ and $B$. A split is \emph{trivial} if either part has only one element. A \emph{split system} is any collection of splits of $[n]$. The elements of $[n]$ are often called the leaves, or the nodes, or the taxa or the terminals. Sometimes for convenience the trivial splits are excluded from systems (or required), here we allow their inclusion unless otherwise stated. A circular order of $[n]$  can be drawn as a polygon with the elements of $[n]$ labeling the $n$ sides. A \emph{circular split system} is a split system for which  a circular order exists such that all the splits can be simultaneously drawn as sides or diagonals of the labeled polygon. Trivial splits are sides of the polygon, separating the label of that side from the rest of $[n],$ while a non-trivial split $A|B$ is a diagonal separating the sides labeled by A and B. For any circular split system we can visualize it by such a polygonal representation, or instead choose a visual representation using sets of parallel edges for each split; these representations are called \emph{circular split networks.} A set of parallel edges \emph{displays} a split $A|B$ if the removal of those edges leaves two connected components with respective sets of  terminals $A$ and $B$. A \emph{bridge} is a single edge that displays a split. We sometimes refer to cutpoint-free portions of a network as \emph{blobs}.
See Example~\ref{splitter} for all these concepts illustrated.\\

Two (drawings of) split networks are considered equivalent if they display the same split system $s$. Different drawings may be found by twisting the graph around the cut-points or bridges, by rotating or flipping.  Again, any circular order $c$ of $[n]$ that allows such a drawing of the split network for $s$ (in either its graph form or the dual polygonal picture) is called consistent with $s$.

Non-negative real numbers, often required to be positive, are assigned to the splits of a split system to make a \emph{weighted split system}. Given a weighted split system $s$ on leaves $[n]$, the \emph{split metric} of $s$ is defined on the leaves $[n]$ by finding the distance $d(i,j)$ as the sum of the weights of splits separating leaf $i$ and $j$. In a split network representing $s$ by weights on the edges, the distance will be the minimal path weight, that is, the sum of the weights of splits traversed on a shortest path from $i$ to $j$.

A metric on $[n]$ can be written as a symmetric matrix $W$, with $W_{ij} = d(i,j).$  A metric is \emph{Kalmanson} if there exists a circular order of $[n]$ such that for any four nodes in circular order $(i,j,k,l)$ we have $$W_{ik}+W_{jl} \ge W_{ij}+W_{kl},$$ and  $$W_{ik}+W_{jl} \ge W_{jk}+W_{il}.$$

Kalmanson metrics were studied first in \cite{ken}, where it is shown that they allow fast solutions of the travelling salesman problem. Their use in phylogenetics is more recent, for more on the connection to circular split systems see  \cite{Pachter2}. The following theorem is exactly what we will need. It is  Proposition 6.10 of \cite{steelphyl}, and has a proof presented in \cite{kleinman}. 

\begin{thm}\label{kal}
A metric on $[n]$ is Kalmanson, with respect to a circular order $c$, if and only if it is the split metric for a unique circular weighted split system $s$. Furthermore (from the proof in \cite{kleinman}) the circular order $c$ is consistent with $s.$   \end{thm}

For a Kalmanson metric, the unique weighted circular split system can be found using various algorithms. We often choose the agglomerative algorithm Neighbor-Net, which is shown to return the exact unique weighted circular split system when the metric is Kalmanson. That algorithm also returns an approximate when the metric fails to be Kalamanson.  \cite{Bryant2007}.\\

A \emph{k-nested} circular planar electrical (or phylogenetic) network $N$ has edges that are part of at most $k$ cycles each. An example of a 1-nested networks is $N$ in Figure~\ref{exrun1pt5}; an example of a 2-nested network is $N$ in Figure~\ref{bridgy}. This concept is introduced for phylogenetic networks in  \cite{gambette-huber}. In \cite{frontiers} the authors show that a 1-nested network has Kalmanson resistance distance. In the next section we extend that theorem to all circular planar electrical networks. \\

\section{Split networks from electrical networks.}\label{splitsec}

 In this paper we introduce some new combinatorial invariants of electrical equivalence, based on the resistance matrix $W.$ Theorem~\ref{bigth} implies that for two circular planar networks if $N\sim N'$ then both have the same weighted split system, denoted $R_w(N) = R_w(N').$ The set of splits in that system is thus also an invariant:  if $N\sim N'$ then both give rise to the same set of splits. %However, the number of sets of splits of $n$ that are achieved in this way is smaller than the number of perfect matchings---there is a projection rather than a bijection.
\begin{example}\label{splitter}

In Figure~\ref{exrun2} the network $N$ from Figure~\ref{exrun1pt5} is shown with resistances on edges. The weighted circular split network $R_w(N)$ is shown; the sum of weights on a shortest path between terminals $i$ and $j$ in $R_w(N)$ equals the effective resistance between those terminals in $N$, and thus the $i,j$ entry of $W(N)$, from above Example~\ref{mats}. There are 7 splits, as shown in the dual polygonal picture of the unweighted split network. For instance the split $\{1,5\}|\{2,3,4\}$ has weight 1/6, the trivial split $\{4\}|\{1,2,3,5\}$ has weight 1/3, and the bridge $\{1,4,5\}|\{2,3\}$ has weight 1/2. Note that the effective resistance from terminal 1 to terminal 5  in $N$ is found in several ways: $$W_{1,5}  = 17/12 \text{ (from the matrix formula in Section~\ref{formulas})}$$ $$
=1+ \frac{(1/2)(1+1+1/2)}{3} \text{ (using Kirchoff's and Ohm's laws on the diagram of $N$ in Figure~\ref{exrun2})}$$ $$
=1/6+1/6+13/12 \text{ (by summing the splits in the shortest path of the circular split network $R_w(N)$.)}$$
\end{example}
 
    \begin{figure}[h]
     \centering
     \includegraphics[width=\textwidth]{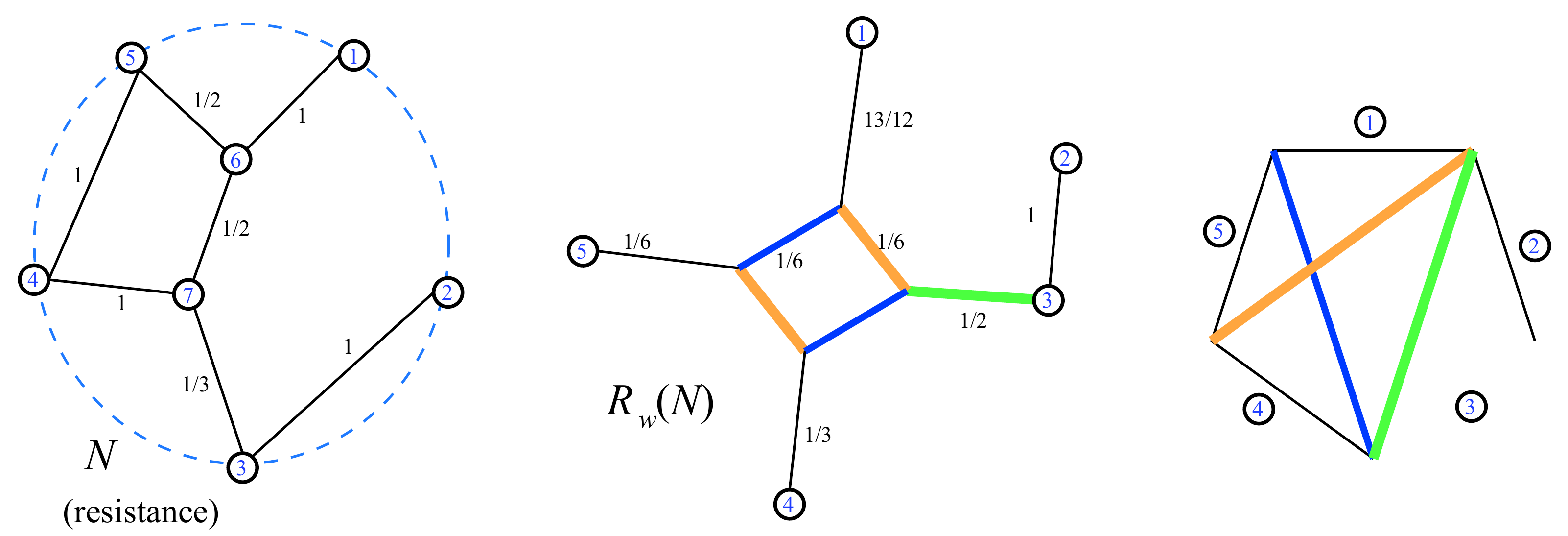}
     \caption{Left: the network $N$ from Figure~\ref{exrun1pt5}, but labeled by resistances. In the center is the corresponding weighted split network $R_w(N)$ found using the resistance matrix $W$. On the right is the dual polygonal picture of the splits. Note that there is no trivial split separating terminal 3 from all the rest.}
     \label{exrun2}
 \end{figure}

\begin{thm}\label{bigth}
If a symmetric matrix $M$ is a response matrix $M=M(N)$ for a connected circular planar network $N$ then the corresponding resistance matrix $W=W(N)$ obeys the Kalmanson condition. Thus each circular planar network (up to equivalence) corresponds to a unique weighted circular split system, denoted $R_w(N)$.

\end{thm} 
\begin{proof}
Given a circular, outer planar network $N$ we take the given circular ordering of its leaves and choose any size 4 circular subsequence of the outer nodes:  $(i,j,k,l)$.

Then let $N_{ijkl}$ be the circular planar network with the same underlying graph as $N$, but with the boundary nodes just those four. Let $N'$ be the Kron reduced network of $N_{ijkl}$ and let $M'$ be the Kron reduction of $M$ with respect to those four nodes. Thus $M'$ is a $4\times 4$ response matrix itself. Since the original network is circular planar then the two distinct circular minors of $M'$ are non-negative.  

 Let $W'$ be the restriction of $W(N)$ to the four chosen nodes $(i,j,k,l).$ Thus $W'$ is the resistance matrix corresponding to $M'.$ Since $W'$ is a restriction we have $W'_{ij} = W_{ij},$ and the same for the rest of the respective off-diagonal entries of $W,W'.$

There is a case to consider for each of the possible Kron reduced graphs $N'$. Most of these cases are simply discussed: the circular \emph{planar} reduced graphs $N'$ are all equivalent to trees or 1-nested networks, so their resistance matrices $W'$ (which are just the restrictions of $W(N)$ to the leaves $i,j,k,l$) are Kalmanson by Theorem 3.1 of \cite{frontiers}. 

If $N'$ is non-circular-planar, then it must contain the edges $\{i,k\}$ and $\{j,l\}$. In that case $N$ must have contained interior paths from $i$ to $k$ and from $j$ to $l$. However, then $N$ must also contain interior paths from $i$ to $j$ and from $k$ to $l$. Thus the only non-circular-planar possibility for $N'$ is the complete graph on the four nodes $i,j,k,l$.
Thus we consider when  $M'$ is the weighted Laplacian of a complete graph $N'$ on 4 vertices, with edge weights the corresponding off-diagonal entries of $M'$. We show $M'$ and the graphs $N, N'$ in Figure~\ref{uno}.
The resistances of each edge were chosen for convenience to be $p,q,r,x,y,z$ so that the conductances are the reciprocals. Thus the non-negative circular minors are:
$$\frac{1}{xq}-\frac{1}{ry} \ge 0; \text{~~~ (this is the circular minor using (i,j;k,l), that is, rows 1,2 and columns 4,3,)}$$ $$ \frac{1}{pz}-\frac{1}{ry}\ge 0; \text{~~~ (this is the circular minor using (j,k;i,l), that is, rows 2,3 and columns 1,4,)}$$
\hspace{1.35in} which imply that $ry - pz\ge 0$ and $ry - qx \ge 0.$ 

\begin{figure}
    \centering
    \includegraphics[width=\textwidth]{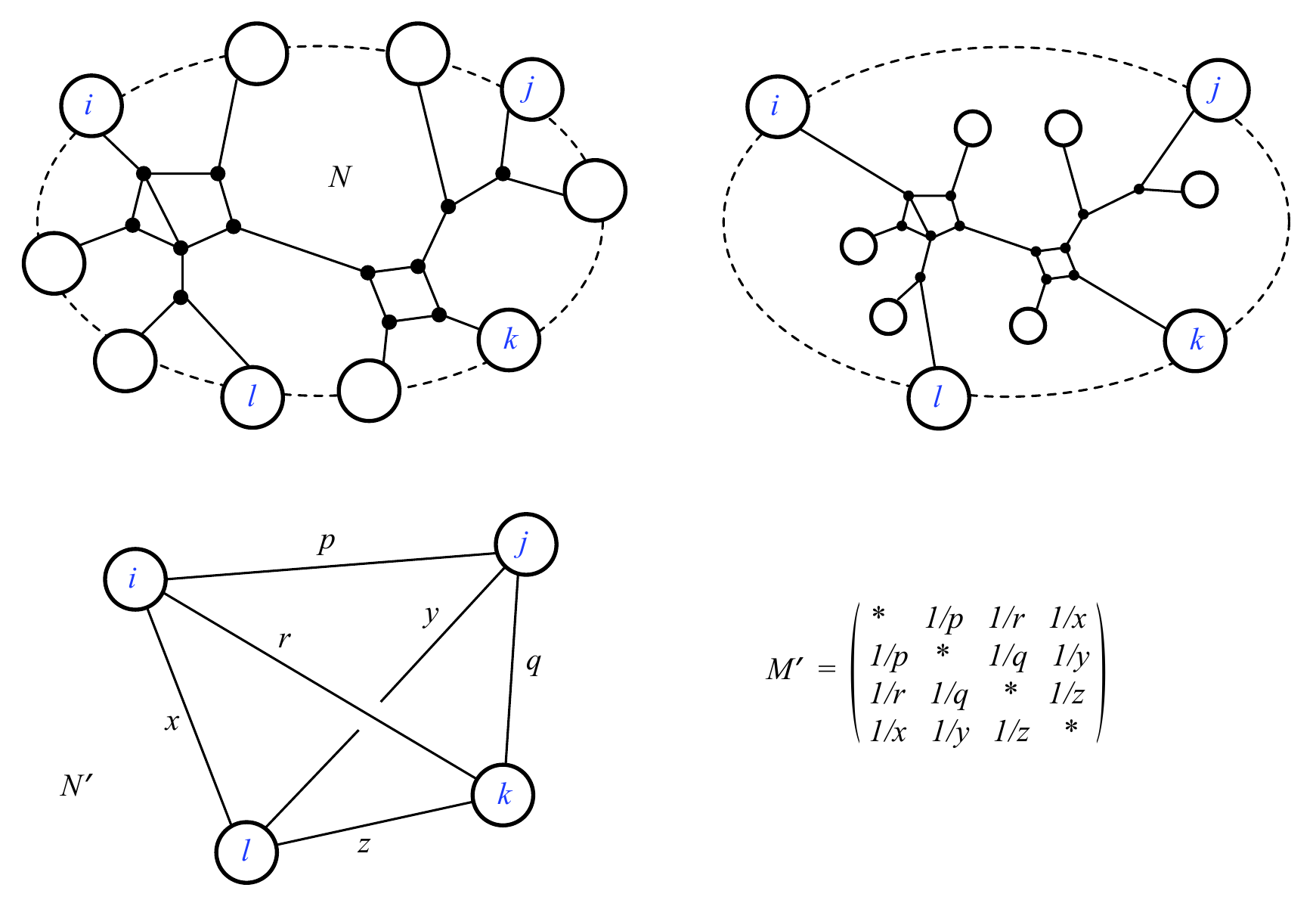}
    \caption{Top: $N$ is an arbitrary circular planar network with a selection of four cyclically ordered nodes. Bottom: the resulting Kron reduction $N'$ is a complete graph (with edges weighted by resistance), and the reduced response matrix  is $M'$. (The asterisks stand for the diagonal entries which make the row sums zero.)}
    \label{uno}
\end{figure}

Next, using either Ohm's law or the pseudoinverse, we calculate: $W_{ik}+W_{jl} - W_{ij}-W_{kl} = $

$$
 \frac{2qx(ry - pz)}{pqr + pqx + pqz + pry + qrx + prz + qry + pxy + pxz + qxy + pyz + qxz + rxy + qyz + rxz + ryz}$$\\
which implies $W_{ik}+W_{jl} - W_{ij}-W_{kl} \ge 0$, so $W_{ik}+W_{jl} \ge W_{ij}+W_{kl}.$\\

Similarly, $W_{ik}+W_{jl} - W_{jl}-W_{ik} = $
$$\frac{2pz(ry - qx)}{pqr + pqx + pqz + pry + qrx + prz + qry + pxy + pxz + qxy + pyz + qxz + rxy + qyz + rxz + ryz}
$$\\
which implies $W_{ik}+W_{jl} - W_{jk}-W_{il} \ge 0$, so $W_{ik}+W_{jl} \ge W_{jk}+W_{il}.$ \end{proof}

 \begin{remark}
 The converse of Theorem~\ref{bigth} does not hold in general. As a counterexample, consider the network $N$ in Figure~\ref{counter}. The conductances as shown give rise to the response matrix $M$ as follows:
 $$
 M=\begin{bmatrix}
-5 & 2 & 1 & 0 & 2\\
2 & -5 & 2 & 1 & 0\\
1 & 2 & -5 & 2 & 0\\
0 & 1 & 2 & -5 & 2\\
2 & 0 & 0 & 2 & -4
\end{bmatrix}
 $$
 The circular pair $(1,2;4,3)$ has the circular determinant -1, which demonstrates that $N$ is non-planar. 
 The corresponding resistance matrix is 
 
 $$
 W=\begin{bmatrix}
0 &    16/51 & 19/51 &  7/17 &  6/17 \\   
      16/51 &  0 &     5/17 & 19/51 & 25/51  \\  
      19/51 &  5/17 &  0 &    16/51 & 25/51  \\  
       7/17 & 19/51 & 16/51 &  0 &     6/17  \\  
       6/17 & 25/51 & 25/51 &  6/17 &  0   
 \end{bmatrix}
 $$

 This $W$ is Kalmanson, as seen by the fact that it has corresponding split network as seen in Figure~\ref{counter}. %Note that this split network does not obey the inequality $bk\ge gh$, as labeled in Figure~\ref{labeled4n5}. Also note that it can be checked that the Kron reduction to any subset of four of five nodes of $N$ is in fact planar. The key is that when node 5 is a leaf, it must either be seen as a sink or a source, but when it is an interior node it can be neither.  
 \end{remark}
 
 \begin{figure}[h]
     \centering
     \includegraphics{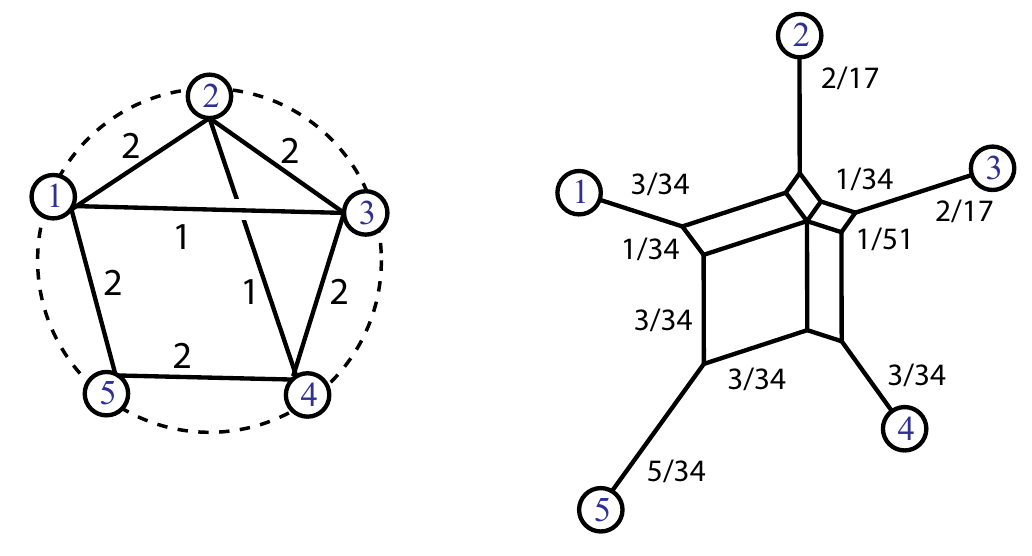}
     \caption{Counterexample: non-circular-planar $N$ on the left and its corresponding circular split network on the right.}
     \label{counter}
 \end{figure}

\begin{definition}
When a circular split system $s$ is found as the image of a circular planar network, $s=R_w(N)$, we say that $s$ is an \emph{electrical circular split system.}
\end{definition}

% Given that $W$ obeys the Kalmanson condition, we have that any restriction of $W$  (constrained to the rows and columns for a subset of the leaves) also obeys the Kalmanson condition, with the restriction to the smaller circular order. Specifically any restriction $W'$ to four leaves in circular order obeys the Kalmanson condition. Therefore, by our calculations in the proof of Theorem~\ref{bigth}, the corresponding 4-leaf Kron reduction to $M'$ has non-negative circular minors. This in turn implies that for any four leaves $(i,j,k,l)$ in the circular order, the network $N_{ijkl}$ is equivalent to one that can be drawn in the plane with those four leaves on the exterior in counting order. \\

\subsection{Bridges preserved}

A bridge of a circular planar network is an edge whose removal disconnects the network. Similarly, a cut-node is a vertex whose removal disconnects the network. Circular split networks corresponding to a circular split system have all the same cut points and bridges. 

\begin{thm}\label{brij}
For any bridge or cut-vertex of $N$ which splits the graph into two components with respective exterior node sets $A,B$, there is a bridge or cut-vertex of the corresponding split network  $s = R_w(N)$ that splits $s = R_w(N)$ into two components with the same sets of leaves $A,B$.
\end{thm}

\begin{proof}
A special case is a cut-vertex $v$ of the boundary of $N.$ In this case, in the Kron reduced network $K(N)$  $v$ is a vertex of at least two cliques. We see that the reduction of the response matrix of the network to the nodes of a single clique in $K(N)$ is the same as finding the submatrix of $M$ using just those nodes of that clique. In terms of resistance, when finding the effective resistance between nodes separated by $v$ the final calculation will be addition of two resistances in series, the sum of the resistances before and after $v$. That is, every distance between terminals separated by the cut-vertex $v$ will satisfy the triangle inequality through $v$ exactly.  Thus a circular split network representing those resistances via addition of paths can be constructed from two smaller circular split networks using the nodes on two sides of $v$, and therefore $v$ itself will be a cut-vertex of the overall split network.

For all other cases, we show that in general $R_w(N)$ does not subtract from the collection of bridges and cut-point nodes of $N.$ We recall that any circular reordering $c$ of the terminals of $N$ which keeps $N$ planar is called \emph{consistent} with $N.$  If $c$ is a circular order consistent with $N$, then $W(N)$ is Kalmanson with respect to that circular order $c.$ Thus $c$ is also consistent with  $R_w(N),$ by Theorem~\ref{kal}.  Therefore, since the set of circular orders consistent with $N$ is determined by twisting around the splits associated to bridges or cut-point nodes of $N$, every bridge or cut-point node of $N$ must correspond to a bridge or cut-point node of $R_w(N),$ else some circular order would no longer be consistent.

\end{proof} 

%Note: To get that both cut-vertices and bridges go to bridges really requires that we choose the a,b,c,d close to the bridge (or for the sake of argument, cutpoint of degree 2) of $s$ ???.

Since interior cut-vertices of $N$ can correspond to bridges of $R_w(N)$, when trying to determine actual bridges of $N$ we will have to inspect each candidate, that is, each bridge of $R_w(N).$

Let a connection of $N$ \emph{avoid} a split $A|B$ if each path in the connection  begins and ends in the same part of the split; $A$ to $A$ or $B$ to $B$. We have the following:

\begin{thm}\label{brijj}
For a split $A|B$  of $[n]$, let there be a bridge or cut-vertex displaying that split in $R_w(N).$ If the connections which avoid that split all exist in $N$, then that split can be represented by a bridge in $N$.
\end{thm}

\begin{proof}
The only connections  of $N $that can use a (non-boundary) cut-vertex as part of their set of paths are the 1-connections, and if the cut-vertex  is expanded to a bridge then that set of 1-connections of $N$ will not change. Thus if the set of connections that avoid a bridge $b$ all exist in $N$, then the total set of connections will be the same with or without the bridge. Therefore since the set of connections determines the network $N$ up to equivalence, we may assume the bridge exists.
\end{proof}

Note that this theorem does not provide necessary conditions for a bridge; there may be bridges which this criteria does not detect. Once a bridge is detected we can use that information to reconstruct the network in pieces. The following will allow the reconstruction process described in Section~\ref{recon}.

\begin{thm}
Given a bridge of $N$, the sub-networks of $N$ separated by that bridge are independently determined by the response matrix $M.$ That is, the inverse problem can be solved separately for those sub-networks. 
\end{thm}

\begin{proof}
Let $A|B$ be the split of $[n]$ displayed by a bridge $b$ of $N.$ Consider the sub-network $N_A$ of $N$ created by deleting $b$ and keeping the component of $N$ with terminals $A$. Let $y$ be the node of $N_A$ that was formerly the end of $b.$ Looking again at the entire network $N$,  we choose one terminal $x$ of $B$ together with all the terminals of $A$ to be kept as the boundary, and demote the remaining terminals of $B-x$ to interior nodes. The resulting network $N'$ will be equivalent to $N_A$ plus the bridge $b$ (now from $x$ to $y$) with a new conductance on $b$ that is the reciprocal of the effective resistance from $x$ to $y.$  Meanwhile the Kron reduction  of $M(N)$ with respect to the terminals of $B-x$ gives the new response matrix $M'(N')$. Solving the inverse problem on $M'$ allows the reconstruction of $N'$ and thus $N_A$.  
\end{proof}
 For example, see the final network in Figure~\ref{exr5four} which is the network $N'$ for the bridge $b$ in Figure~\ref{covernet} that displays the split $A|B = \{5,6,7\}|\{1,2,3,4,8,9,10\}.$ In Figure~\ref{exr5four} we chose $x=1$ and used the Kron reduction of $M$ to reconstruct the graphs of $N'$. 
 
\subsection{Obstruction}

In \cite{frontiers} it is conjectured that for every weighted 2-nested network there is a weighted 1-nested network with matching resistance distance. Here we show that is not the case.

\begin{thm}\label{sym}
Let a circular split network $s$ have two non-trivial, weighted, crossing, splits with equal weights: $w((A\cup B)|(C\cup D))=a$ and $w((A\cup D)|(B\cup C)) =a$; (where the unions are disjoint) and the two related non-trivial, non-crossing splits also equally weighted: $w((A\cup B\cup C)|D) = b$ and $w((A\cup D\cup C)|B) = b.$ An example is on the right hand side of Figure~\ref{absym}. If $s=R_w(N)$ is the image of a 1-nested network $N$ then $a=b.$
\end{thm}
\begin{proof}
In \cite{frontiers} it is shown that for a resistance weighted 1-nested network $N$, the resistance distance is Kalmanson and gives rise to a faithfully phylogenetic circular split network. A split in $R_w(N)$ corresponding to two non-adjacent edges of a cycle in the network  $N$ with resistance weights $p,q$ has weight $pq/t$ where $t$ is the sum of the edge weights in that cycle. Let that be the first of our two crossing splits, and the other be displayed by edges with weights $x,y.$ Then the weights of the splits in $R_w(N)$ are $pq/t=a$, $xy/t=a,~~px/t=b$ and $yq/t=b.$  Since $t > 0,$ these equations imply $a=b.$
\end{proof}

\begin{figure}
    \centering
    \includegraphics[width =4.25in]{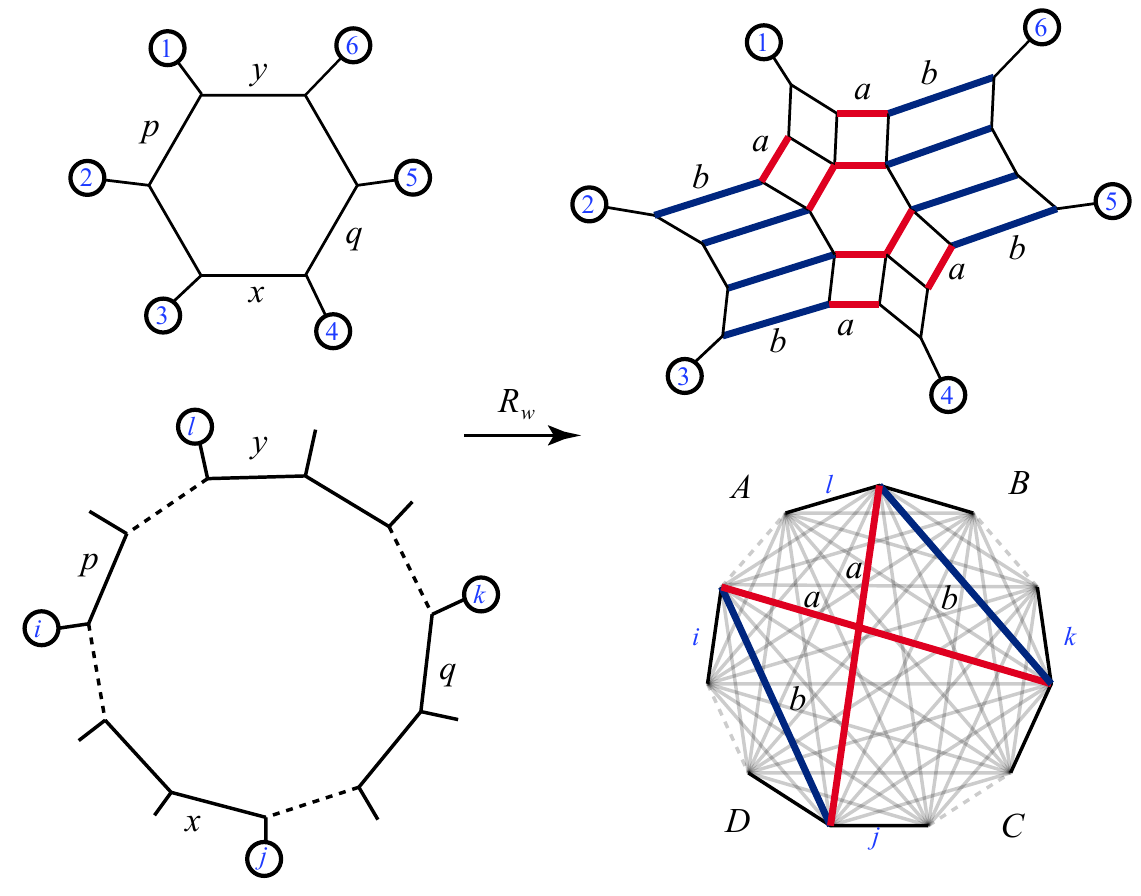}
    \caption{Illustration for Theorem~\ref{sym}. The networks $N$ are shown at left, and two split networks on the right. At the top is the smallest possible example, with a cycle of size six. At the bottom is the general case, with the split network on the right drawn in the dual polygonal representation.}
    \label{absym}
\end{figure}

So we are able to look for obstructions that prevent a circular split network from corresponding to a 1-nested circular planar network. For instance, Figure~\ref{bridgy} shows an example of a circular split network that arises from a 2-nested network $N.$  Theorem~\ref{sym} shows that this circular split network cannot arise from a 1-nested circular planar network.
  
  \begin{figure}
    \centering
    \includegraphics[width=\textwidth]{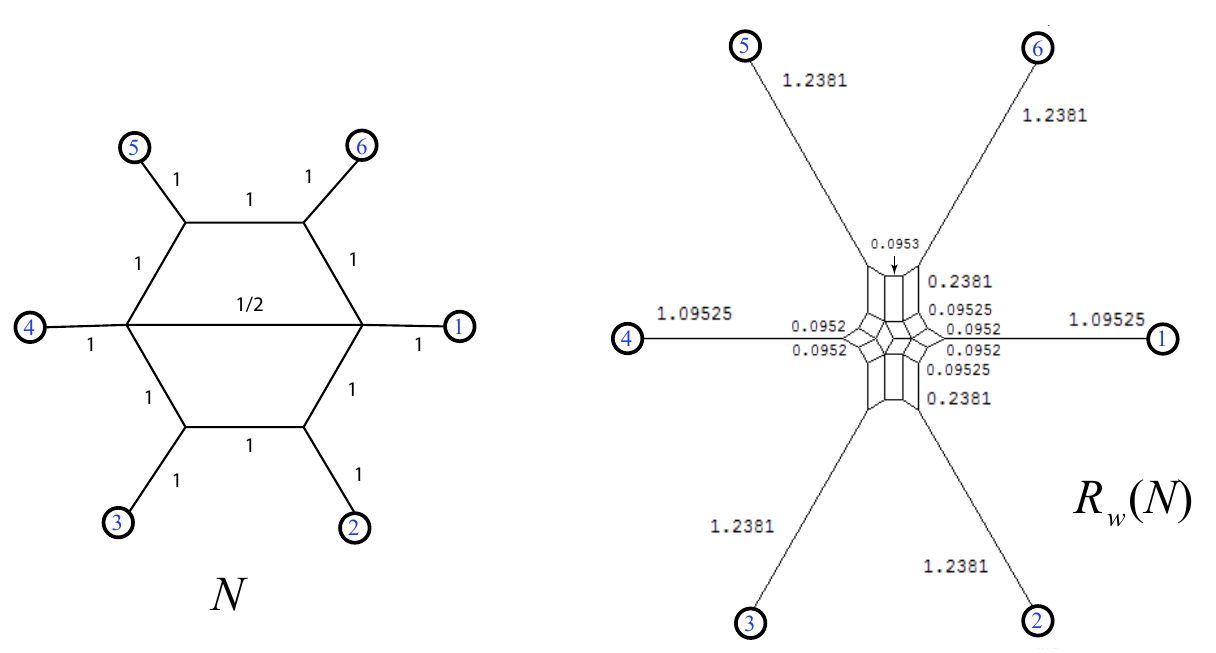}
    \caption{A 2-nested circular planar network $N$ with edges weighted by conductance. Its associated split network with weighted splits shows that it cannot be equivalent to a 1-nested network.}
    \label{bridgy}
\end{figure}
\newpage

\section{Reconstruction}\label{recon}
Given an experimentally determined matrix $M$ or $W$, here are the steps to recognize the existence of, and reconstruct if possible, a circular planar network $N$ (inside the black box) such that $M$ is found by testing  conductances (or $W$, by testing resistances) of the $n$ exposed terminals of $N$. At first we label the terminals arbitrarily to construct the experimental $M$, or $W$. Recall that either matrix can be found from the other using Lemma~\ref{formulas}.  We list the steps here and then label them in the same way in Example~\ref{big}.
\begin{enumerate}
    \item[\textbf{Step 1.}] Check to see if $M$ is a valid weighted Laplacian: symmetric, with the off-diagonal entries non-negative, the diagonals non-positive, and the rows and columns each summing to zero. Constructing the graph of this Laplacian will also determine whether it is connected. At this point we can subdivide the problem via terminals that are cut-points of the Kron reduction itself. For instance, in Figure~\ref{exrun1pt5}, we could restrict to just the terminals 1,3,4,5 since the last, terminal 2, is seen to be separated by the cutpoint terminal 3. \\ 
    \item[\textbf{Step 2.}] If $M$ is indeed the Laplacian of a connected network, then we proceed to test whether that network can be equivalent to a circular planar network $N$. The  question is whether the boundary nodes can be arranged in a circle on the plane with $N$ embedded in the disk. To begin this check, we may find the resistance matrix $W$ associated with $M$. If $W$ is not Kalmanson, then $M$ cannot be planar. If $W$ is Kalmanson with respect to some circular ordering $c$ of the $n$ terminals, then we must still check the circular minors of $M$ to verify that the desired planar network $N$ exists. Note that the original matrix $M$ may not have circular minors all non-negative, but if $W$ is Kalmanson with respect to $c$ then the entries of $W$ and $M$ are easily rearranged  to match the circular order $c$. Then the rearranged $M$ can be checked to see if it has all non-negative circular minors. 
    One way to check $W$ for the Kalmanson property is to run the algorithm neighbor-net on $W$, which detects the order $c$ and finds the unique circular split system $s$ corresponding to $W$ (and to $W_c$), if those exist. We may need to check multiple potential compatible circular orders $c$ to decide if any give a valid circular planar network.
    
     (Note that if we discover some circular ordering $c$ of $[n]$  other than the counting order,  then we can re-start by re-labelling the leaf-nodes by $[n]$ via the clockwise counting order for convenience.) \\
    
    \item[\textbf{Step 3.}]  Drawing a diagram of a split network that displays the split system $s$  immediately reveals the bridges and cut points of the desired network $N$, (but not which is which). This allows us to produce a good first approximation to $N,$ without knowing the interior of the bridge-free portions. Note that a bridge in $s$ can represent either a bridge or a cutpoint of $N.$ We need to determine which, since only the bridges can be used to subdivide our problem. The criteria is that a candidate bridge $b$ must not increase the connections that exist in $M.$  To be sure of this, we must know that $M$ already reflects all the connections which are subdivided by the split displayed by $b$. That is, $M$ must contain positive minors for every connection that could be blocked by shrinking the bridge $b$ to a single cut-point. Deciding might require checking some number of circular minors of $M$ for positivity, and thus decrease our efficiency. If we knew from the beginning that our network contained no cutpoints displaying splits unless a bridge displayed that split, we could proceed much more quickly. (In fact this condition is practical: it is often assumed for phylogenetic networks that the interior nodes are degree 3, since speciation and hybridization usually occur as binary operations.)\\
    %\item At most two strands can connect the edges of any two blobs connected by a bridge. For $N$ a level-1 network (proven, and conjectured for all levels), the pairs of parallel strands in each blob are determined by the equal products of splits.
    %\item Since all the strands in each blob of $N$ are determined by the parallel strands, we can proceed to construct $N$ itself. (proof: there is only one strand diagram for $k$ nodes and no parallels.) The strand diagram divides the disk into regions which can be colored with two colors. The boundary nodes all fall into regions of one of the two colors, and we place interior nodes in the remaining regions of that same color. Then the edges are found by connecting nodes through each strand crossing. 
    
    \item[\textbf{Step 4.}] The graph reconstruction algorithm of Curtis and Morrow, from Chapter 9 of \cite{curtisbook}, can be performed on each  non-trivial-bridge-free  portion (blob). To do so, we start by reducing $W$ to a smaller matrix on each blob in turn. For the new boundary terminals of the subnetwork that makes each blob, we must choose appropriate original terminals from $[n].$ For each bridge that touches the blob in our split network, we may choose any terminal that is separated from the blob by that bridge.  \\
    
%    \item Note that for $N$ level-1, the CSN completely determines $N$ already...but how can we detect level-1 from the matrix or the splits? We have some obstructions.
    
    \item[\textbf{Step 5.}] After reconstructing each blob the graph of the entire network is found, up to equivalence, by reconnecting our new subnetworks with bridges. Finally, the conductances of the edges are found from the matrix $M_c$ using either the Curtis-Morrow algorithm from \cite{curtisbook} or the algorithm of Kenyon and Wilson, in \cite{kenyon}. Note that the tree-like portion will be immediately solved:  conductance of a tree edge, between two nodes neither of which is part of a cycle, is just the reciprocal of the weight (resistances) of the split.
\end{enumerate}

\begin{example}\label{big}

We show an example for reconstructing a larger network. Starting with $M$ (experimental) we use the formula to get $W$ or vice-versa, and then use Neighbor-net to find the splits. First here is a given response matrix. (We calculated it from an example $L$ using the Schur complement.) Note that there are rounding errors present at every stage; this example should demonstrate both the steps in a situation with perfect measurements and computations, and show several places where approximation can take over should perfection not be the case.
\end{example}

\begin{comment}
$W =$
$$
\begin{bmatrix}
  0 & 2.9693 & 3.4693 & 2.3359 & 2.2863 & 2.3967 & 3.2538 & 1.9612 & 2.7189 & 1.6528\\
  2.9693 &    0 & 1.5000 & 1.0333 & 2.2370 & 2.3474 & 3.2045 & 2.0386 & 3.0102 & 2.1701\\
  3.4693 & 1.5000   &  0 & 1.5333 & 2.7370 & 2.8474 & 3.7045 & 2.5386 & 3.5102 & 2.6701\\
  2.3359 & 1.0333 & 1.5333 &    0 & 1.6036 & 1.7140 & 2.5712 & 1.4052 & 2.3769 & 1.5368\\
  2.2863 & 2.2370 & 2.7370 & 1.6036 &    0 & 0.3052 & 1.2922 & 1.0019 & 2.0914 & 1.3299\\
  2.3967 & 2.3474 & 2.8474 & 1.7140 & 0.3052 &    0 & 1.2468 & 1.1123 & 2.2018 & 1.4403\\
  3.2538 & 3.2045 & 3.7045 & 2.5712 & 1.2922 & 1.2468 &    0 & 1.9694 & 3.0589 & 2.2975\\
  1.9612 & 2.0386 & 2.5386 & 1.4052 & 1.0019 & 1.1123 & 1.9694 &    0 & 1.6485 & 0.9263\\
  2.7189 & 3.0102 & 3.5102 & 2.3769 & 2.0914 & 2.2018 & 3.0589 & 1.6485 &    0 & 1.4132\\
  1.6528 & 2.1701 & 2.6701 & 1.5368 & 1.3299 & 1.4403 & 2.2975 & 0.9263 & 1.4132 &    0
\end{bmatrix}
$$
\end{comment}
\textbf{Step 1.}
$M =$
$$\displaystyle{\begin{bmatrix}
\frac{  -391}{585} & \frac{ 139}{6763} & \frac{139}{13526} & \frac{223}{2170} & \frac{ 72}{1373} & \frac{ 30}{1373} & \frac{ 10}{1373} & \frac{467}{3141} & \frac{315}{11626} & \frac{223}{578 } \\\\
\frac{  139}{6763} & \frac{-1541}{1283} & \frac{439}{1099} & \frac{929}{1405} & \frac{265}{8779} & \frac{ 63}{5009} & \frac{ 21}{5009} & \frac{311}{5009} & \frac{ 48}{4937} & \frac{801}{20026} \\\\
\frac{  139}{13526} & \frac{439}{1099} & \frac{-585}{731} & \frac{ 727}{2199} & \frac{113}{7487} & \frac{ 63}{10018} & \frac{21}{10018} & \frac{311}{10018} & \frac{24}{4937} & \frac{401}{20051} \\\\
\frac{  223}{2170} & \frac{929}{1405} & \frac{727}{2199} & \frac{-476}{281} & \frac{ 756}{5009} & \frac{315}{5009} & \frac{105}{5009} & \frac{556}{1791} & \frac{240}{4937} & \frac{4007}{20036} \\\\
\frac{  72}{1373} & \frac{265}{8779} & \frac{113}{7487} & \frac{756}{5009} & \frac{-2296}{585} & \frac{2011}{702} & \frac{ 607}{2106} & \frac{493}{912} & \frac{1019}{14606} & \frac{703}{3662  } \\\\
\frac{  30}{1373} & \frac{ 63}{5009} & \frac{ 63}{10018} & \frac{315}{5009} & \frac{2011}{702} & \frac{-2071}{569} & \frac{ 258}{569} & \frac{ 166}{737} & \frac{ 257}{8841} & \frac{271}{3388 } \\\\
\frac{  10}{1373} & \frac{ 21}{5009} & \frac{ 21}{10018} & \frac{105}{5009} & \frac{607}{2106} & \frac{258}{569} & \frac{-483}{569} & \frac{ 166}{2211} & \frac{ 84}{8669} & \frac{180}{6751 } \\\\
\frac{  467}{3141} & \frac{311}{5009} & \frac{311}{10018} & \frac{556}{1791} & \frac{493}{912} & \frac{ 166}{737} & \frac{ 166}{2211} & \frac{-953}{798} & \frac{ 309}{889} & \frac{1151}{1364 } \\\\
\frac{  315}{11626} & \frac{48}{4937} & \frac{ 24}{4937} & \frac{240}{4937} & \frac{1019}{14606} & \frac{257}{8841} & \frac{ 84}{8669} & \frac{309}{889} & \frac{-193}{254} & \frac{ 582}{1147  } \\\\
\frac{  223}{578} & \frac{ 801}{20026} & \frac{401}{20051} & \frac{4007}{20036} & \frac{703}{3662} & \frac{271}{3388} & \frac{180}{6751} & \frac{1151}{1364} & \frac{582}{1147} & \frac{-2020}{1263}
\end{bmatrix}}$$

\textbf{Step 2.} Next we calculate the resistance matrix using $$W(M) = \left((-M)^{\dagger}\right)_DJ+J\left((-M)^{\dagger}\right)_D-2(-M)^{\dagger}. $$ 

$W =$
$$\begin{bmatrix}
 { 0} & \frac{2414}{813} & \frac{1523}{439} & \frac{ 904}{387} & \frac{1765}{772} & \frac{1148}{479} & \frac{3423}{1052} & \frac{3593}{1832} & \frac{1857}{683} & \frac{757}{458} \\\\
 \frac{ 2414}{813} & { 0} & \frac{ 3}{2} & \frac{31}{30} & \frac{ 1973}{882} & \frac{1284}{547} & \frac{1426}{445} & \frac{2801}{1374} & \frac{2068}{687} & \frac{ 523}{241} \\\\
 \frac{ 1523}{439} & \frac{ 3}{2} & { 0} & \frac{  23}{15} & \frac{ 1207}{441} & \frac{ 914}{321} & \frac{1567}{423} & \frac{1744}{687} & \frac{4823}{1374} & \frac{1287}{482} \\\\
\frac{  904}{387} & \frac{  31}{30} & \frac{ 23}{15} & {0} & \frac{ 619}{386} & \frac{ 917}{535} & \frac{1373}{534} & \frac{ 912}{649} & \frac{ 801}{337} & \frac{ 919}{598} \\\\
 \frac{ 1765}{772} & \frac{1973}{882} & \frac{1207}{441} & \frac{ 619}{386} & { 0} & \frac{  47}{154} & \frac{ 199}{154} & \frac{1598}{1595} & \frac{4256}{2035} & \frac{653}{491} \\\\
 \frac{ 1148}{479} & \frac{1284}{547} & \frac{ 914}{321} & \frac{ 917}{535} & \frac{  47}{154} & { 0} & \frac{  96}{77} & \frac{  426}{383} & \frac{2706}{1229} & \frac{2993}{2078} \\\\
\frac{  3423}{1052} & \frac{1426}{445} & \frac{1567}{423} & \frac{1373}{534} & \frac{ 199}{154} & \frac{  96}{77} & {0} & \frac{ 837}{425} & \frac{1661}{543} & \frac{ 363}{158} \\\\
 \frac{ 3593}{1832} & \frac{2801}{1374} & \frac{1744}{687} & \frac{ 912}{649} & \frac{1598}{1595} & \frac{426}{383} & \frac{ 837}{425} & { 0} & \frac{ 755}{458} & \frac{1697}{1832} \\\\
\frac{  1857}{683} & \frac{2068}{687} & \frac{4823}{1374} & \frac{801}{337} & \frac{4256}{2035} & \frac{2706}{1229} & \frac{1661}{543} & \frac{ 755}{458} & { 0} & \frac{ 920}{651} \\\\
\frac{  757}{458} & \frac{ 523}{241} & \frac{1287}{482} & \frac{ 919}{598} & \frac{ 653}{491} & \frac{2993}{2078} & \frac{363}{158} & \frac{1697}{1832} & \frac{920}{651} & { 0} 
\end{bmatrix}
$$

\begin{comment}
$M =$
$$
\begin{bmatrix}
  -0.6684 & 0.0206 & 0.0103 & 0.1028 & 0.0524 & 0.0218 & 0.0073 & 0.1487 & 0.0271 & 0.3858\\
  0.0206 & -1.2011 & 0.3995 & 0.6612 & 0.0302 & 0.0126 & 0.0042 & 0.0621 & 0.0097 & 0.0400\\
  0.0103 & 0.3995 & -0.8003 & 0.3306 & 0.0151 & 0.0063 & 0.0021 & 0.0310 & 0.0049 & 0.0200\\
  0.1028 & 0.6612 & 0.3306 & -1.6940 & 0.1509 & 0.0629 & 0.0210 & 0.3104 & 0.0486 & 0.2000\\
  0.0524 & 0.0302 & 0.0151 & 0.1509 & -3.9248 & 2.8647 & 0.2882 & 0.5406 & 0.0698 & 0.1920\\
  0.0218 & 0.0126 & 0.0063 & 0.0629 & 2.8647 & -3.6397 & 0.4534 & 0.2252 & 0.0291 & 0.0800\\
  0.0073 & 0.0042 & 0.0021 & 0.0210 & 0.2882 & 0.4534 & -0.8489 & 0.0751 & 0.0097 & 0.0267\\
  0.1487 & 0.0621 & 0.0310 & 0.3104 & 0.5406 & 0.2252 & 0.0751 & -1.1942 & 0.3476 & 0.8438\\
  0.0271 & 0.0097 & 0.0049 & 0.0486 & 0.0698 & 0.0291 & 0.0097 & 0.3476 & -0.7598 & 0.5074\\
  0.3858 & 0.0400 & 0.0200 & 0.2000 & 0.1920 & 0.0800 & 0.0267 & 0.8438 & 0.5074 & -1.5994
\end{bmatrix}
$$
\end{comment}

\textbf{Step 3.}

Now the split system $s$ is found directly from $W$. We used the Neighbor Net algorithm, as implemented in SplitsTree \cite{split, hb}. The split system is shown as a split network in Figure~\ref{exr5split}. In this example, the clockwise order of $[n]$ is correct (all the circular minors of $M$ with respect to that order are non-negative, at least approximately.) There are 24 splits overall. The splits labeled $a$ and $b$ are potential bridges of the network.

  \begin{figure}[h]
     \centering
     \includegraphics[width=\textwidth]{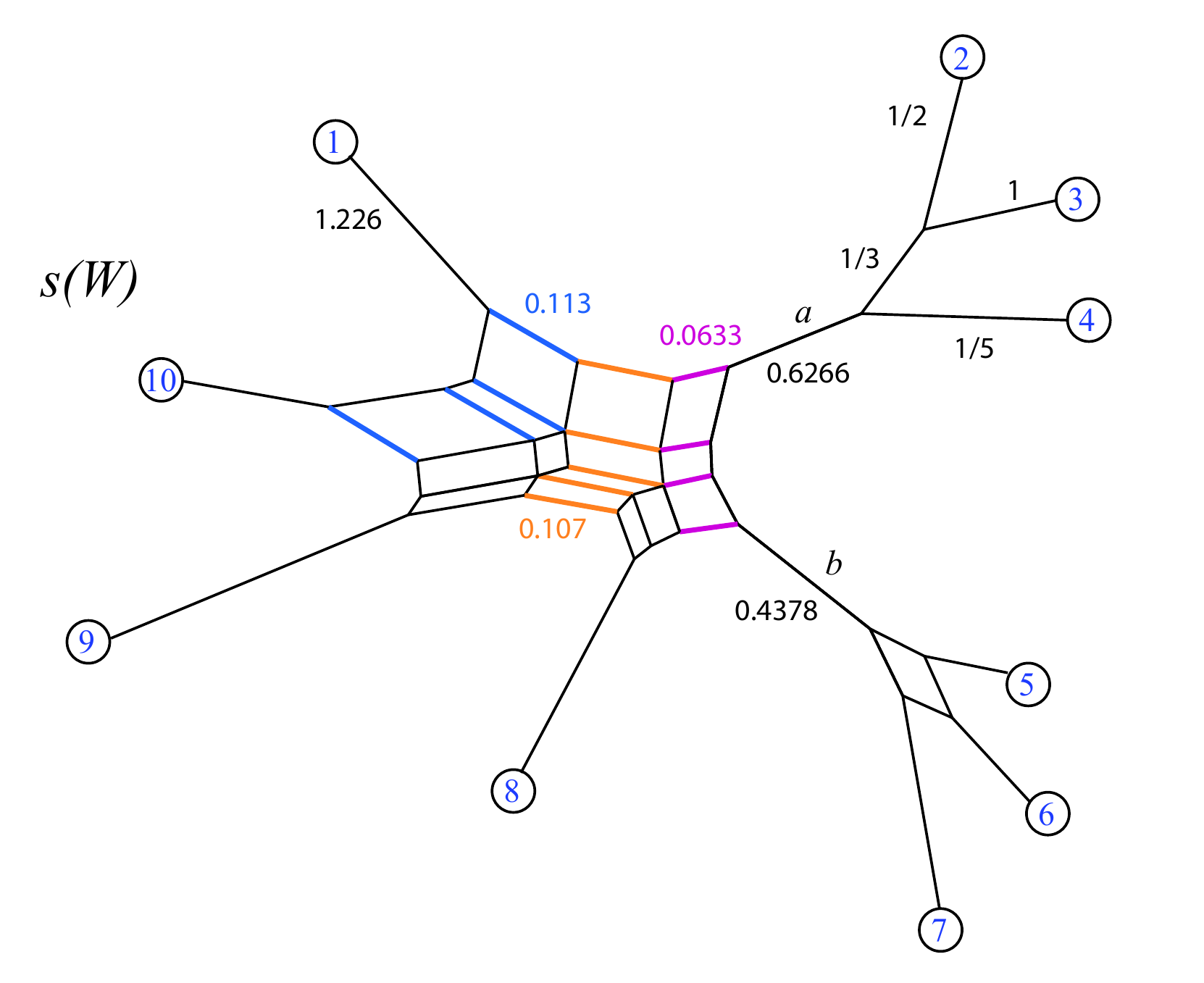}
     \caption{The split system for $W$, shown as a split network. Some of the split weights are shown. The bridges of $s$ labeled $a$ and $b$ are potential bridges of the network $N$.}
     \label{exr5split}
 \end{figure}
 
The split network bridge structure allows us to isolate the blobs. First we check the bridges $a$ and $b$ to see if they are real network features. We can check the larger connections with terminals closer to the bridge in order to minimize the workload. 
For potential bridge  $b$ the connections that suffice are those connections involving paths from node  5 to 7, but involving 
at most one node from $\{2,3,4\}$, since that cluster is also separated by a bridge (or cut-point).
Thus we check the circular minor of $M$ for the circular pair (1, 2, 5; 9, 8, 7). 
$$
det\begin{bmatrix}
    315/11626 & 467/3141 & 10/1373\\
    48/4937 & 311/5009 & 21/5009\\
    1019/14606 & 493/912 & 607/2106
\end{bmatrix} = \frac{27}{473351}
$$

Similarly for potential bridge $a$ we check the circular minor  (10, 1, 2; 8, 7, 4).

$$
det\begin{bmatrix}
    1151/1364 & 180/6751 & 4007/20036\\
    467/3141 & 10/1373 & 223/2170\\
    311/5009 & 21/5009 & 929/1405
\end{bmatrix} = \frac{17}{13246}
$$

Both those connections exist, since the minors are non-zero. Thus both $a$ and $b$ are actual bridges, since their proposed existence does not increase the set of connections.\\

\textbf{Step 4.}

Thus we can focus on smaller graphs. First we find the blob of size 6. We choose a terminal connected by an internal path to each of its corners: here we 
chose terminals $\{1,2,5,8,9,10\}$.   We relabel those as $1,\dots,6$. We let $P = W|\{1,2,5,8,9,10\}$, resistance restricted to those terminals. Then $S$ is the response matrix 
associated to $P$:
\begin{comment}
$P =$
$$
\begin{bmatrix}
     0 & 2.9693 & 2.2863 & 1.9612 & 2.7189 & 1.6528\\
  2.9693   &  0 & 2.2370 & 2.0386 & 3.0102 & 2.1701\\
  2.2863 & 2.2370 &    0 & 1.0019 & 2.0914 & 1.3299\\
  1.9612 & 2.0386 & 1.0019 &    0 & 1.6485 & 0.9263\\
  2.7189 & 3.0102 & 2.0914 & 1.6485 &    0 & 1.4132\\
  1.6528 & 2.1701 & 1.3299 & 0.9263 & 1.4132 &    0
\end{bmatrix}
$$

$S =$
$$
\begin{bmatrix}
  -0.6653 & 0.0770 & 0.0770 & 0.1186 & 0.0194 & 0.3734\\
  0.0770 & -0.5699 & 0.1578 & 0.1836 & 0.0248 & 0.1267\\
  0.0770 & 0.1578 & -1.1711 & 0.6513 & 0.0694 & 0.2157\\
  0.1186 & 0.1836 & 0.6513 & -1.7491 & 0.2257 & 0.5699\\
  0.0194 & 0.0248 & 0.0694 & 0.2257 & -0.7862 & 0.4470\\
  0.3734 & 0.1267 & 0.2157 & 0.5699 & 0.4470 & -1.7326
\end{bmatrix}
$$
$Q = $
$$
\begin{bmatrix}
    0 & 2.2863 & 2.3967 & 3.2538\\
  2.2863 &  0 & 0.3052 & 1.2922\\
  2.3967 & 0.3052 &  0 & 1.2468\\
  3.2538 & 1.2922 & 1.2468  &  0
\end{bmatrix}
$$
$T=$
$$ \begin{bmatrix}  -0.4453 & 0.2863 & 0.1193 & 0.0398\\
  0.2863 & -3.6126 & 2.9948 & 0.3316\\
  0.1193 & 2.9948 & -3.5855 & 0.4715\\
  0.0398 & 0.3316 & 0.4715 & -0.8428
\end{bmatrix}$$
\end{comment}

$$P =\begin{bmatrix}
 0 & 2414/813 & 1765/772 & 3593/1832 & 1857/683 &  757/458  \\
    2414/813 & 0 & 1973/882 & 2801/1374 & 2068/687 &  523/241 \\ 
    1765/772 & 1973/882 & 0 & 1598/1595 & 4256/2035 & 653/491  \\
    3593/1832 & 2801/1374 & 1598/1595 &   0 &  755/458 & 1697/1832 \\ 
    1857/683 & 2068/687 & 4256/2035 & 755/458 & 0 &  920/651  \\
     757/458 &  523/241 &  653/491 & 1697/1832 & 920/651 & 0 
\end{bmatrix}$$

$$S =\begin{bmatrix}
   -1159/1742 & 373/4844 & 292/3793 & 428/3609 &  67/3461 & 2377/6366  \\
     373/4844  &   -1390/2439 & 287/1819 & 474/2581 &  67/2699 & 211/1666 \\ 
     292/3793 & 287/1819  &   -1307/1116 & 3394/5211 & 163/2350 & 401/1859  \\
     428/3609 & 474/2581 & 3394/5211  &   -1520/869 &  809/3585 & 485/851 \\ 
  67/3461 &  67/2699 & 163/2350 & 809/3585 & -592/753 &  274/613  \\
    2377/6366 & 211/1666 & 401/1859 & 485/851 &  274/613 & -674/389  
\end{bmatrix}$$

 Next we use the algorithm given by Curtis and Morrow in Chapter 9 of \cite{curtisbook}, starting with the response matrix $S$. We number points on the circle as seen in Figure~\ref{exr5strand}: the $2n$ \emph stub points $t_i$ begin just before terminal 1, in clockwise order, and are just before and after each terminal. The $2n$ \emph{cutting points} $x_i$ begin with $x_1$ just before $t_1$, and are between successive stub points; landing on the terminals and on the midpoints between terminals.  On the left of  Figure~\ref{exr5strand} we show the matrices (note that sometimes more rows are needed, as in Chapter 9 of \cite{curtisbook}.) First we find the matrix $MR$ = MaxRespected($S$),  where entry $MR(i,j)$ is the size $k$ of the largest 
$k$-connection that \emph{respects the cut} from $x_i$ to $x_j$. That is, the largest $k$ such that the circular minor of size $k$ is positive and the connection $(A,B)$ has terminals $A$ all entirely on one side of the line from $x_i$ to $x_j$ (when the line begins or ends on a terminal those terminals don't count). For instance $MR(1,7)=3$ because the circular pair (1, 2, 3; 6, 5, 4) has circular minor in $S$ of:
$$
det\begin{bmatrix}
    2377/6366 & 67/3461 & 428/3609\\
    211/1666 & 67/2699 & 474/2581\\
    401/1859 & 163/2350 & 3394/5211
\end{bmatrix} = \frac{16}{18659}
$$
In contrast $MR(2,8)=2$ since the largest size connection that can respect the cut is of size 2, and indeed the pair (2, 3; 6, 5) has a positive minor. One more example: $MR(3,9)=2$ since the minor for the pair (2, 3, 4; 1, 6, 5) is:
$$
det\begin{bmatrix}
   373/4844   & 211/1666 &  67/2699 \\ 
     292/3793 & 401/1859 & 163/2350   \\
     428/3609  & 485/851 &  809/3585 
\end{bmatrix} = -0.000000002
$$
Unfortunately, that last calculation is only approximately zero; the computer has accumulated some rounding errors by the time we are looking at minors of $S$. This demonstrates the fact that Curtis and Morrow point out in the introduction to \cite{curtisbook}: the problem is ill-posed in that it is sensitive to rounding in computation.

Then the matrix $RE$ = Re-entrants($S$) = NumTerminals(6)-MaxRespected($S$). NumTerminals($n$) is the matrix with row 1 given by $0,0,1,1,2,2,\dots,n-1,n-1$ and row $i$ for $i>1$ given by the same string shifted right by $\lfloor{i/2}\rfloor+1$ places (with initial zeroes.)  
We draw the a strand from $t_i$ to $t_j$ when row $i+1$ differs first from row $i$ in column $j+1$ of $RE$.  Notice that all we need 
to find are 5 strands, the last is determined. 

  \begin{figure}[h]
     \centering
     \includegraphics[width=\textwidth]{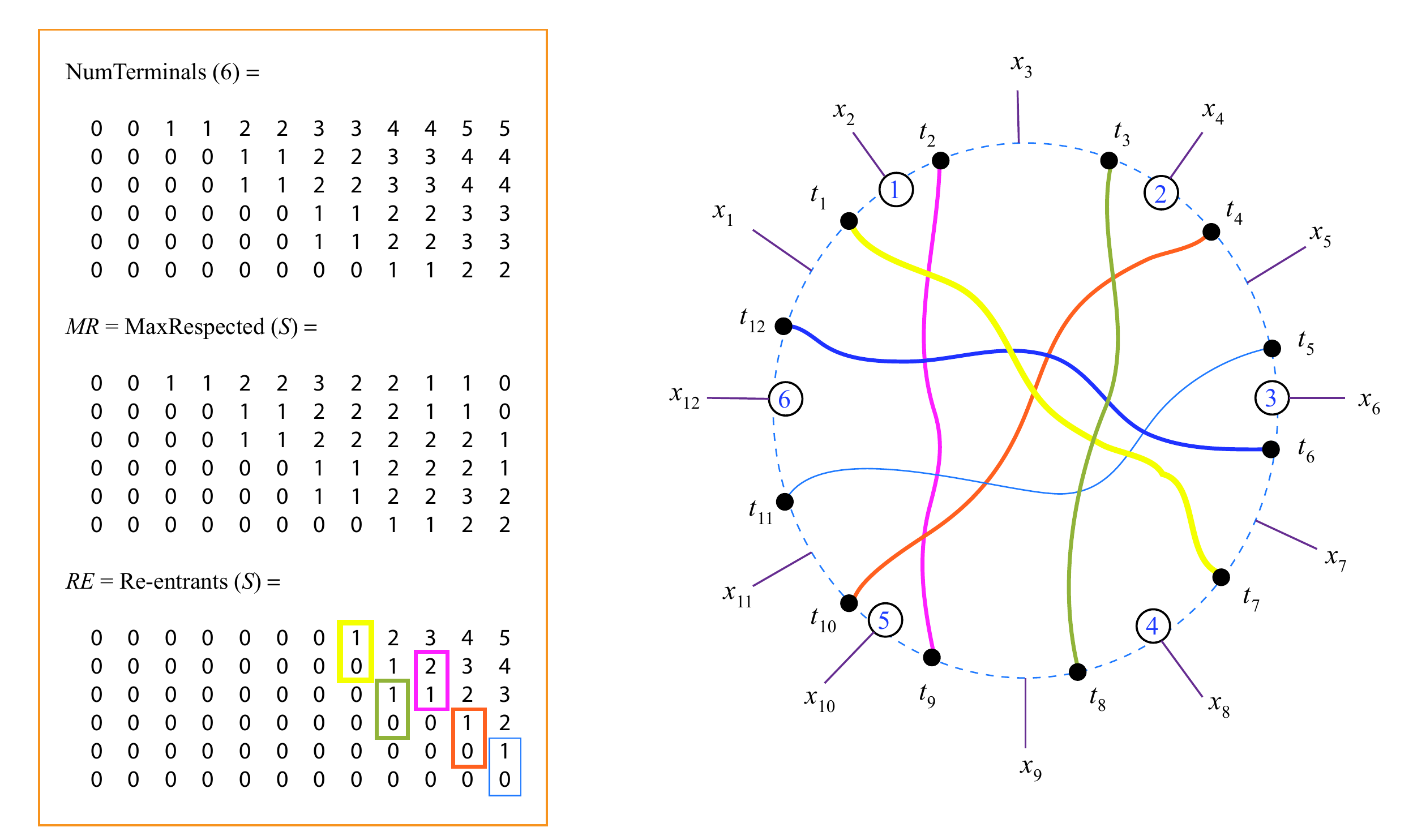}
     \caption{The matrix $S$ is used to find the matrix of re-entrants, which is used in turn to find the perfect matching. Strands are drawn to represent that matching, taking care to avoid lenses. }
     \label{exr5strand}
 \end{figure}

Next, in Figure~\ref{exr5blob},  we draw the local network for the size-6 blob by placing a node in each shaded region (the first nodes we place around the circle are the terminals). Then the graph of the subnetwork is reconstructed.
 
  \begin{figure}[h]
     \centering
     \includegraphics[width=\textwidth]{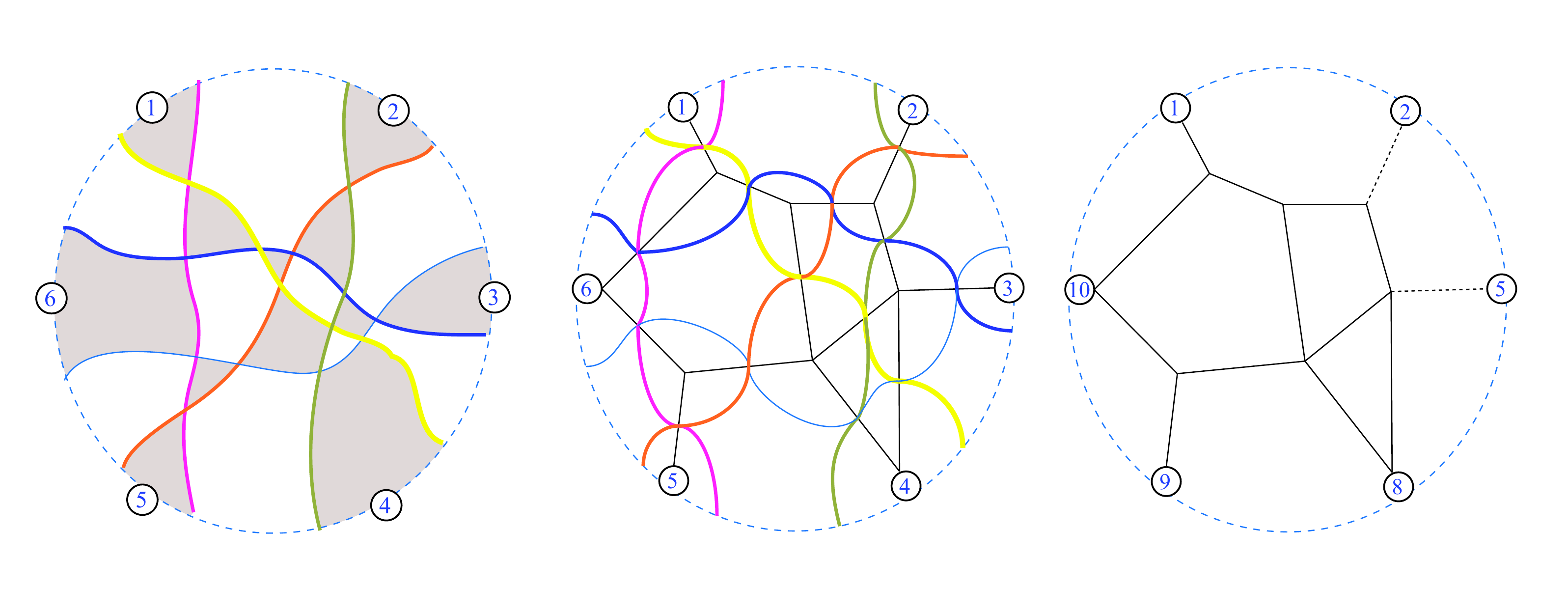}
     \caption{ The checker-board shading on the left is a guide for the nodes of our sub-network. Each shaded region is a node, and nodes are connected by an edge when their shaded regions meet at a crossing of the strands. (The strands form the medial graph.) We show the sub-network in the third picture with dashed edges going to the eventual terminals: those must be replaced by other subnetworks.}
     \label{exr5blob}
 \end{figure}

The process is repeated for the size 4 blob in Figure~\ref{exr5four}, using the matrix $Q$ which is $W$ restricted to terminals $\{1,5,6,7\}$, and $T$
which is the response matrix associated to $Q$.

$$Q =\begin{bmatrix}
  0 & 1765/772 & 1148/479 & 3423/1052  \\
    1765/772 & 0 &   47/154 &  199/154  \\
    1148/479 &   47/154 & 0 &   96/77    \\
    3423/1052 & 199/154 &   96/77 & 0
\end{bmatrix}$$

$$T =\begin{bmatrix}
    -985/2212 & 298/1041 & 211/1769 & 178/4477  \\
     298/1041 &    -2294/635 & 1141/381 &  379/1143 \\ 
     211/1769 & 1141/381 & -3962/1105 & 521/1105  \\
     178/4477 & 379/1143 & 521/1105 & -547/649
\end{bmatrix}$$

  \begin{figure}[h]
     \centering
     \includegraphics[width=\textwidth]{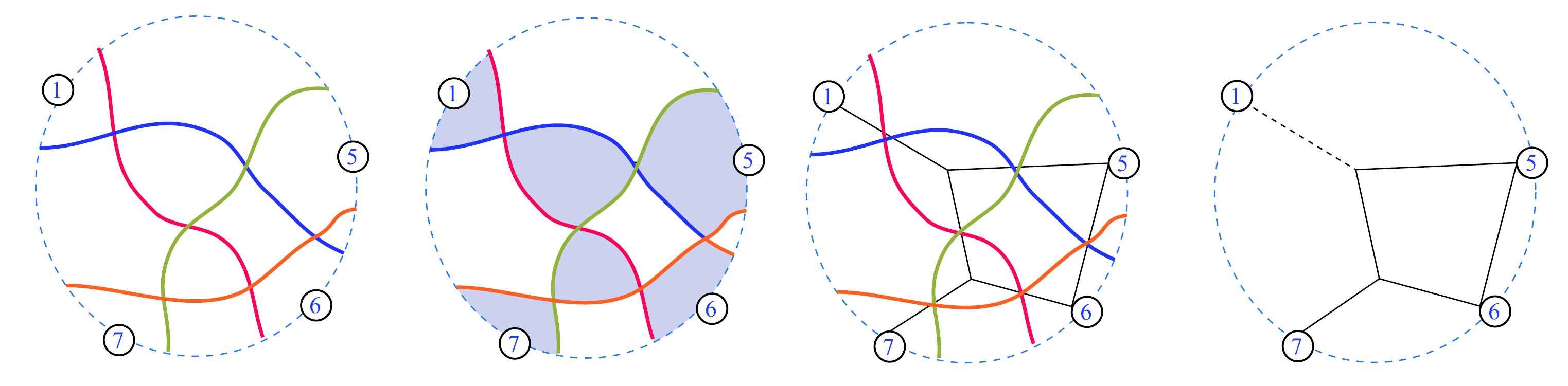}
     \caption{The strands, shading, and resulting subnetwork for the smaller blob of $N.$}
     \label{exr5four}
 \end{figure}

\textbf{Step 5.} 

Finally in Figure~\ref{exr5tot} the two blobs and the tree-like portion are rejoined with the bridges $a$ and $b$ to see the full network. Once the graph is drawn, an algorithm such as Kenyon and Wilson's from \cite{kenyon} can be chosen to recover the conductances from $M$.

  \begin{figure}[h]
     \centering
     \includegraphics[width=\textwidth]{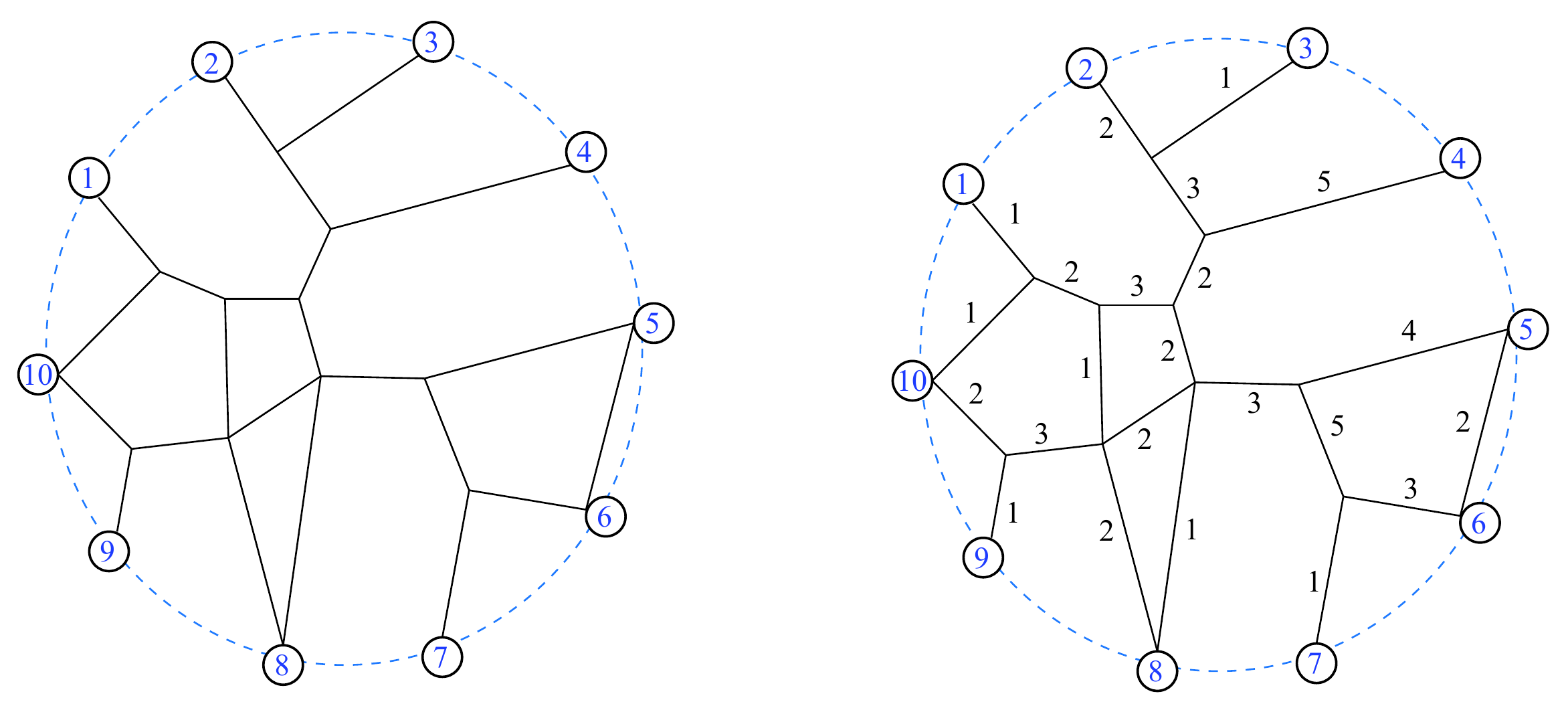}
     \caption{The final reconstruction of $N$, from the sub-networks found above. The conductances here are the ones used to create the example. }
     \label{exr5tot}
 \end{figure}

\section{Spaces, Polytopes and future directions}\label{new}

The Symmetric Travelling Salesman polytope STSP($n$) has as vertices all the circular orders of $[n].$ Relatedly, the space of Kalmanson metrics (and thus circular planar split systems CSN($n$)) is usually studied all at once, for all circular orderings of $[n]$, as in \cite{scalzo}. The space of split systems for a particular order, like the clockwise ordering $1,\dots,n$ that we see in this paper, is a top-dimensional cell in the total space $CSN(n)$. There are  $(n-1)!/2$  top-dimensional cells, one for each circular order, that are glued along shared faces. The structure of both the fundamental cells and the global space of circular split networks is studied in \cite{dev-petti} and \cite{terhorst}. There is certainly also  a global structure for the total space of circular planar electrical networks, with a top-dimensional cell for each circular order.  A given circular planar electrical network $N$ may have several consistent re-orderings of its terminals, each also circular planar. Permuting its matrix $M$ to one of those orderings will preserve the positivity in circular minors. Thus $N$ will be shared
by several cells of the total space. We would like to describe the topology of that total space.

However,  most studies now use the terminology \emph{space of circular planar electrical networks} to refer to only a certain given circular ordering, a single cell of the total space.  The original work of Curtis and Morrow \cite{curtis1} gave us $\Omega_n$, the space of response matrices for a given circular order of $[n].$ Alman, Lian and Tran \cite{alman} described the graded poset $EP_n$ of equivalence classes of circular planar (unweighted) graphs,  indexed by sets of critical graphs, or by certain perfect matchings. Lam \cite{lam1} extended this poset, and compactified the space of circular planar networks (for the given circular ordering) by allowing all perfect matchings on $[2n].$ The matchings which identify terminals correspond to infinite conductance, a shorting of the circuit between boundary nodes so that they are considered a single node. This compactified space of circular planar electrical networks is called $E_n$, with cells indexed by all the matchings $P_n.$ Lam shows that $E_n$ is embedded as a linear slice of the totally non-negative Grassmannian.  Postnikov shows in \cite{postnikov} that the non-negative Grassmannian has points corresponding to the matrix $M$ of boundary measurements from weighted \emph{directed} circular planar networks. 
In \cite{hersh1} Hersh and Kenyon show that the posets $P_n$ are shellable, and that each $P_n$ can be realized as the face poset of a regular CW complex. This same regularity property has been conjectured (see \cite{lam-ball} and \cite{lam2}) for the compactified space of circular planar electrical networks $E_n$. Hersh and Kenyon point out that their shellability result suggests that the stratified spaces $E_n$ may be regular CW complexes with each cell closure homeomorphic
to a closed ball. Can  a better understanding of the map $R_w$ taking networks to split systems settle this question?

In \cite{durell} and \cite{scalzo} the authors show that the form of a 1-nested phylogenetic (or electrical) network can be found by linear programming on a series of polytopes. These include the (BME) Balanced Minimum Evolution polytopes of phylogenetic trees, and a series of polytope families BME($n,k$)  for $n\ge 3$ and  $0\le k \le n-3. $ When $k=0$ these are the Symmetric Travelling Salesman polytopes STSP($n$). However, all of the 1-nested network polytopes BME($n,k$) are found nested inside STSP($n$), so linear programming on the latter using the resistance metric $W$ as a linear functional can reveal the bridge structure of any circular planar network. (The Neighbor Net algorithm is a greedy approach to the same problem.) In \cite{scalzo} there is shown a Galois connection between the network-faces of the symmetric travelling salesman polytope and the cells of the Kalmanson complex. Can the asymmetric traveling salesman polytope play a similar role for directed networks?

\bibliographystyle{amsplain}
\bibliography{phylokron}{}

\end{document}